\documentclass{elsarticle}
\newcommand*{\mytitle}{Identification of stochastic operators}
\newcommand{\coloneqq}{=}
\usepackage{amssymb, amsmath, amsthm}
\usepackage{mathtools}
\usepackage[Symbolsmallscale]{upgreek} 
\usepackage{bm} 
\usepackage{mathrsfs}

\usepackage[OT2,OT1]{fontenc}
\DeclareSymbolFont{cyrletters}{OT2}{wncyr}{m}{n}
\DeclareMathSymbol{\Shah}{\mathalpha}{cyrletters}{"78}
 
 \DeclareFontFamily{U} {MnSymbolC}{}
\DeclareFontShape{U}{MnSymbolC}{m}{n}{
  <-6> MnSymbolC5
  <6-7> MnSymbolC6
  <7-8> MnSymbolC7
  <8-9> MnSymbolC8
  <9-10> MnSymbolC9
  <10-12> MnSymbolC10
  <12-> MnSymbolC12}{}
\DeclareSymbolFont{MnSyC} {U} {MnSymbolC}{m}{n}
\DeclareMathSymbol{\mysmallbox}{\mathord}{MnSyC}{"68}
\usepackage{paralist} % for \inparaenum
\usepackage{xspace} 
\usepackage{subfig} 
\usepackage{csquotes} % for \enquotes
\usepackage{hyperref}
\hypersetup{
pdfkeywords = {sampling, identification, stochastic, reconstruction, sounding, radar, channel sounding, spreading function, scattering function, time-varying operators},  
pdfauthor = {G\"{o}tz E. Pfander, Pavel Zheltov},  
pdftitle = {\mytitle},
pdfpagelayout = OneColumn, pdfstartview = FitH,  colorlinks, citecolor=blue, linkcolor=blue, urlcolor=blue}
\theoremstyle{plain}
\newtheorem{theorem}{Theorem}[section]

\usepackage{aliascnt} 
\newaliascnt{lemma}{theorem}
\newtheorem{lemma}[lemma]{Lemma}
\aliascntresetthe{lemma}

\newaliascnt{proposition}{theorem}

\aliascntresetthe{proposition}

\newaliascnt{corollary}{theorem}

\aliascntresetthe{corollary}

\usepackage{chngcntr}

\counterwithout{figure}{section}
\theoremstyle{definition}
\newaliascnt{definition}{theorem}
\newtheorem{definition}[definition]{Definition}
\aliascntresetthe{definition}

\newtheorem*{remark}{Remark}
\newcommand*{\R}{\mathbb{R}}
\newcommand*{\N}{\mathbb{N}}
\providecommand{\Z}{\mathbb{Z}}
\newcommand*{\C}{\mathbb{C}}

\renewcommand{\S}{\mathcal{S}} % Schwartz space
\newcommand*{\FT}{\mathcal{F}}

\newcommand*{\Ball}[1]{\mathcal{B}(#1)}

\renewcommand{\d}{\:\mathrm{d}}

\renewcommand{\phi}{\varphi}
\newcommand*{\eps}{\varepsilon} %{\varepsilonit}

\DeclareMathOperator{\supp}{supp}

\DeclareMathOperator{\vol}{vol\,}

\renewcommand{\Vec}{\operatorname{vec}}

\renewcommand*{\implies}{\quad \Rightarrow \quad}

\DeclareMathOperator{\Dom}{Dom}
\DeclareMathOperator{\Sq}{Sqrt}
\DeclareMathOperator{\Covariance}{Cov}
\newcommand*{\Expectation}{\mathbb{E}}
 
\newcommand*{\eqms}{\mathrel{\stackrel{\scriptstyle m.s.}{=}}}
\newcommand*{\leqsim}{\lesssim}

\newcommand*{\geqsim}{\gtrsim}

\newcommand*{\conj}[1]{\overline{#1}}
\newcommand*{\COV}{\Covariance} 
\newcommand*{\EXP}{\Expectation\,}

\DeclarePairedDelimiter\ceiling{\lceil}{\rceil}
\DeclarePairedDelimiter\floor{\lfloor}{\rfloor}
\DeclarePairedDelimiter\abs{\lvert}{\rvert}
\DeclarePairedDelimiter\norm{\lVert}{\rVert}

\DeclarePairedDelimiter\braces{\lbrace}{\rbrace}
\DeclarePairedDelimiter\brackets{\lbrack}{\rbrack}
\DeclarePairedDelimiter\paren{\lparen}{\rparen}
\DeclarePairedDelimiter\ip{\langle}{\rangle}

\DeclarePairedDelimiter\dualp{\langle}{\rangle}
\newcommand*{\Set}[2]{\braces{#1 \text{ such that } #2}}
\newcommand*{\SetBig}[2]{\braces[\Big]{#1 \text{ such that } #2}}

\newcommand*{\E}[1]{\Expectation\braces*{#1}}
\newcommand*{\littleo}[1]{\mathit{o}\paren*{#1}}
\newcommand*{\bigO}[1]{\mathit{O}\paren*{#1}}

\newcommand*{\el}[2][{}]{\ell^{#2}(\Z^{#1}\times \Z^{#1})}
\newcommand*{\eval}{\big\vert}

\newcommand*{\epi}[1]{\:e^{2\pi i #1}}
\newcommand*{\empi}[1]{\:e^{-2\pi i #1}}

\newcommand*{\acorr}{autocorrelation\xspace}

\newcommand*{\nbhd}{neighborhood\xspace} 

\newcommand*{\bupu}{bounded uniform partition of unity\xspace} 
\newcommand*{\nnd}{positive semi-definite\xspace}

\renewcommand{\O}{\Omega}

\newcommand*{\Minf}{{M^\infty}} 
\newcommand*{\Mone}{{M^1}} 
\newcommand*{\Moneone}{\Mone \widehat{\otimes} \Mone} 
\newcommand*{\w}{\omega}

\renewcommand{\k}{{\boldsymbol{\kappa}}}
\newcommand*{\steta}{{\boldsymbol{\eta}}}

\newcommand*{\h}{{\boldsymbol{h}}}
\newcommand*{\OO}{-\O/2,\O/2}
\renewcommand*{\L}{\Z_L}       %{\{0,\dotsc, L-1\}} 

\newcommand*{\f}{\boldsymbol{f}}   
   
\renewcommand*{\c}{\boldsymbol{c}}   
\newcommand*{\g}{\boldsymbol{g}}

\newcommand{\tntn}{t,\nu;t'\!,\nu'}

\newcommand*{\OPW}{\operatorname{OPW}}
\newcommand*{\St}{St}
\newcommand*{\StOPW}{\St\!\OPW}
\newcommand*{\SMinf}{\St\Minf}

\newcommand*{\Sl}{\St\ell}
 
\newcommand*{\Slinf}{\Sl^\infty}

\newcommand*{\Id}{\operatorname{Id}} 
\newcommand*{\Hcal}{\mathcal{H}}
\newcommand*{\Mcal}{\mathcal{M}}
 
\newcommand*{\Trans}{T}   %

\newcommand*{\sumZ}[1]{\sum_{{#1}\in\Z}}

\newcommand*{\RR}{\R^2}  %{\R\times\hat{\R}} 
\newcommand*{\rect}{\Box} %{\APLbox} %{\operatorname{rect}}

\newcommand*{\KN}{\sigma}
\newcommand*{\kernel}{\kappa}
\newcommand*{\RV}{\operatorname{RV}} 
\newcommand*{\Zak}{\mathcal{Z}}
\renewcommand{\H}{\boldsymbol{H}}

\newcommand*{\tensoratom}[4]{\conj{{#1}_{#2}} \otimes {#3}_{#4} }

\newcommand*{\GG}{\tensoratom{G}{}{G}{}} 

\newcommand*{\RHf}{R_{\scriptscriptstyle \H f}}
\newcommand*{\Reta}{R_{\scriptstyle \steta}}

\newcommand{\klkl}{k,l,k'\!,l'}
\newcommand*{\klmn}{k,l,m,n}
\newcommand*{\klmnp}{k'\!,l'\!,m'\!,n'}
\newcommand*{\MM}{\Mcal \otimes \conj{\Mcal}}  %{\conj{\Mcal}\otimes\Mcal} 
\newcommand*{\Rf}{R_{\boldsymbol{f}}}
 
\newcommand{\placeholder}{\mathord{\,\cdot\,}} 
\newcommand*{\e}{\boldsymbol{e}} % temporary name for a discrete stoch proc.
\renewcommand*{\P}{\mathscr{P}} 
 
\newcommand*{\M}{\mathcal{M}} 
 
\newcommand*{\stoch}[1]{\bm{\mathcal{#1}}}
 
\newcommand*{\etarect}{\eta_{\mysmallbox}}
\newcommand*{\kl}{k,l}
\newcommand*{\klp}{k'\!,l'}
\newcommand*{\klg}{k,l,\gamma}
\newcommand*{\klgp}{k'\!,l'\!,\gamma'}

\newcommand*{\jjp}{{(j,j')}}
\newcommand*{\klmnjj}{k,l,m_\jjp,n_\jjp}
\newcommand*{\klmnjjp}{k'\!,l'\!,m'_\jjp,n'_\jjp}
\newcommand{\mnmn}{m,n,m'\!,n'} 
\newcommand*{\pqpq}{p,q,p'\!,q'}

\newcommand*{\sitau}{\sigma,\tau}
\newcommand*{\sitaup}{\sigma'\!,\tau'} 
\newcommand*{\sitausitau}{\sitau,\sitaup}

\newcommand*{\txtx}{\tau,\xi,\tau'\!,\xi'}
\newcommand*{\SITAU}{\begin{bmatrix} \sigma & \tau & \sigma' & \tau' \end{bmatrix}} 
 
\newcommand*{\PQ}{\begin{bmatrix} p & q & p' & q'\end{bmatrix}} 
\newcommand{\myvec}[1]{\mathrm{#1}}  %\mathsf, mathfrak
\newcommand{\jj}{\myvec{j}}
\newcommand{\kk}{\myvec{k}}
\renewcommand{\ll}{\myvec{l}}
\newcommand{\mm}{\myvec{m}}
\newcommand{\nn}{\myvec{n}}
\newcommand{\pp}{\myvec{p}}
\newcommand{\qq}{\myvec{q}}
\renewcommand{\ss}{\myvec{\upsigma}}
\renewcommand{\tt}{\myvec{\uptau}}
\newcommand{\ff}{\myvec{f}}
\newcommand{\PP}{\myvec{P}}
\renewcommand{\gg}{\myvec{g}}

\newcommand{\Aleph}{J}  %{\aleph} 
\newcommand{\boldphi}{\boldsymbol{\phi}}
\renewcommand{\phi}{\varphi}
\newcommand{\myindent}{\hspace*{1cm}}

\newcommand{\Eta}{\widetilde{\mathrm{H}}} 
\newcommand{\badH}{\boldsymbol{H}} 
\newcommand{\Gscr}{\mathcal{G}}
\newcommand{\GGscr}{\conj{\mathcal{G}} \otimes \mathcal{G}}

\renewcommand*{\a}{\boldsymbol{a}} 
\newcommand*{\Ra}{R_{\a}} % \Cov{\a,\a}
\newcommand*{\B}{B}
\newcommand{\Mtilde}{\widetilde{M}}
\newcommand{\Ltilde}{\tilde{\Lambda}}

\newcommand{\size}[1]{\norm{#1}_{\infty}}
\newcommand{\coeff}[1]{(1+\size{#1})}
\newcommand*{\MMM}{M} 
\newcommand{\pambda}{{\lambda'}} 
\newcommand{\Em}[1]{M_{#1}}    %{\Mtilde_{\Gamma_{#1}}} 
\begin{document}
\date{\today}
\title{\mytitle}
\author[j]{G\"otz~E.~Pfander\fnref{fn1}}
\ead{g.pfander@jacobs-university.de}

\author[j]{Pavel~Zheltov\fnref{fn1}\corref{cor1}}
\ead{p.zheltov@jacobs-university.de}
\address[j]{School of Engineering and Science, Jacobs University Bremen, 28759 Bremen, Germany}
\fntext[fn1]{G.~E.~Pfander and P.~Zheltov acknowledge funding by the Germany Science Foundation (DFG) under Grant 50292 DFG PF-4, Sampling Operators.}
\cortext[cor1]{Corresponding author}
\fntext[subj]{MSC2010: Primary 42B35, 60G20, 42C40;  Secondary 94A20, 47G10} 
\begin{abstract}
Based on the here developed functional analytic machinery we extend the theory of operator sampling and identification to apply to operators with stochastic spreading functions.
We prove that identification with a delta train signal  is possible for a large class of stochastic operators that have the property that the autocorrelation of the spreading function is supported on a set of 4D volume less than one and this support set does not have a defective structure.  
In fact, unlike in the case of deterministic operator identification, the geometry of the support set has a significant impact on the identifiability of the considered operator class. 
Also, we prove that, analogous to the deterministic case, the restriction of the 4D volume of a support set to be less or equal to one is necessary for identifiability of a stochastic operator class. 
\end{abstract}
\maketitle

\listoffigures
%%%%%%%%%%%%%%%%%%%%%%%%%%%%%%%%%%%%%%%%%%%
\section{Introduction}\label{sec:intro} 
In the fields of wireless communication and radar and sonar acquisition, a sounding signal that is known both to the sender and to the observer is sent into the channel in order to determine the characteristics of the channel from the received echo. Similarly, in control theory, the problem of identifying a system from the output to a given input is called system identification. 

Both deterministic and stochastic operator identification have their roots in the works of Kailath \cite{Kailath} and Bello \cite{BelloMeas}. They suggested a criterion for identifiability based on the \emph{spread} of the operator, defined as the area of the support of the spreading function. In \cite{KozPfa, PfaWal} the criterion that it is necessary and sufficient that the spread must be less than one for a deterministic operator to be identifiable has been theoretically verified and justified, giving new life to this engineering dogma. 
The universal boundary of \emph{one} for the spread of the deterministic operator to be identifiable is closely related to the Heisenberg uncertainty principle of quantum mechanics. This connection is evident in the time-frequency analysis nature of the proofs given in \cite{KozPfa, PfaWal}, as these rely on the representation theory of the Weyl-Heisenberg group. 

The utility of weighted delta trains as a theoretical tool for the study of identification and sampling theory stems from their position as infinite bandwidth unbounded temporal support sounding signals. 
Recently, more practical identifier signals for a class of channels with a parametric model on the channel structure have been discovered \cite{Bajwa, KraPfa}. In another development, it was shown that a stiff requirement to know the support of the spreading function (precisely the set whose area must be less than one) prior to sounding can be removed by using compressed sensing techniques \cite{HB, PfaWalPreprint}.

Here, we continue to rely on weighted delta trains as identifiers to develop a parallel theory of identification of operators with stochastic spreading functions. 
Since stochastic operators include deterministic operators as special case, it is tenable to suppose that in some form the restriction on the spread to be less than one retains its relevance. However the rules of the game change, as the object to recover, the spreading function, belongs to a much larger class of objects. 
The common strategy to circumvent these difficulties is to decrease the complexity of the channel by requiring the spreading function to have a degenerate form, stationarity in both the time and frequency variables, which corresponds to a WSSUS channel. The study is hereby reduced to the recovery of the so-called scattering function, a deterministic function in two variables, defined below in \eqref{eq:scattering}.
Even in this simplified setting, the communication engineering literature still seems to accept the insight of Bello that the area of the support of the scattering function is a necessary requirement for the identifiability of the operator \cite{BelloMeas}. In fact, even this characterization of the fading properties of a channel is discarded in favor of a simpler yet \emph{spread  factor} given by the area of the minimum rectangle that encompasses the support of the spreading function (and hence, scattering function) in the time-frequency plane. This rule of thumb is perpetuated in the classical books as recent as the monograph of Proakis \cite{Proakis}. 

In \cite{OPZ}, we argue that as in the deterministic case, the condition for the spread factor to be less than one is sufficient in the case of identifiable WSSUS channels, and establish the direct applicability of time-frequency analysis techniques of Kozek, Pfander and Walnut in this simplified stochastic setting. 
In \cite{PfaZh02}, we assume functional analytic results proven here and give a detailed analysis of the general case of stochastic operator sampling with a fully stochastic spreading function. The work in \cite{PfaZh02} extends sampling results for operators and discusses their applications.  A surprising corollary shows that using weighted delta trains as identifiers for WSSUS channels allows the recovery of the scattering function from the autocorrelation of the received signal irrespectively of the area of the support of the scattering function. 

Here, we settle the question of the identifiability of stochastic operators with a general necessary and sufficient \autoref{thm:Phist.is.MM}. It turns out that the volume of the set $M$, the support of the stochastic spreading function, alone is not enough for identifiability of the corresponding operator. The geometry of the set $M$ plays an important role for the possibility of identification. \\[1ex]

The popular and herein used model for channels and linear time-variant (LTV) systems is 
\begin{equation}\label{eq:Hf}
(H f)(x) = \iint \eta(t,\nu)  \: M_\nu  T_t \:  f(x) \d t \d\nu,
\end{equation}
where $T_t$ is a \emph{time-shift} by $t$, that is, $T_t f(x) = f(x-t)$, $t\in\R$, and $M_\nu$ is a \emph{frequency shift} or \emph{modulation} given by $M_\nu f(x) = \epi{\nu x}\, f(x)$, $\nu\in \R$. Taking Fourier transforms, it follows that $\widehat{M_\nu f} (\xi) = \widehat{f}(\xi-\nu) = T_\nu \widehat{f}(\xi)$ for all $\xi\in \R$. The function $\eta(t,\nu)$ is called the \emph{(Doppler-delay) spreading function} of $H$.
Classically, the domain and codomain of $H$ are taken to be the Lebesgue space of square integrable functions $L^2(\R)$ or the discrete finite dimensional space $\C_L$. More generally, $f(x)$, $\eta(t,\nu)$ and $g(x) = H f(x)$ can be elements in spaces of generalized functions, such as the space of tempered distributions $\S'(\R^d) $, the continuous dual of the space $\S(\R^d) $ of infinitely differentiable rapidly decaying functions. 

It is common that models of wireless channels and radar environments  take the stochastic nature of the medium into account. In such models, the spreading function $\steta(t,\nu)$ or the sounding signal $\f(x)$ in \eqref{eq:Hf}, or both, are random processes (that will henceforth be denoted by boldface letters) such that, for example, every sample of the spreading function $\steta(t,\nu; \w)$ belongs to one of the spaces mentioned above. 
In this paper we consider only the spreading function to be stochastic, leaving the sounding signal completely deterministic. 

Usually, the operator is split into the sum of its deterministic portion, representing the mean behavior of the channel, and its zero-mean stochastic portion that represents the noise and the environment. 
We assume that this decomposition has already taken place and focus on operators with purely stochastic zero-mean spreading functions. For a treatment of the deterministic part, we refer to \cite{KozPfa, PfaWal, Pfander}. 

The statistic that presents the most interest in this setting is the \emph{autocorrelation} of the spreading function
\[ 
\Reta(\tntn) = \E{\steta(t,\nu)\, \conj{\steta(t',\nu')}},
\]
and we will pursue the goal of determining $\Reta$ from $\RHf$, that is, $\Reta$ from the autocorrelation of the stochastic channel output $\H f$ (see \autoref{defn:id.3}). 
The time-varying case most studied in the literature assumes the special form 
\begin{equation}\label{eq:scattering}
\Reta(\tntn) = C_\steta (t,\nu) \, \delta(t-t') \, \delta(\nu-\nu'). 
\end{equation}
Such operators are referred to as \emph{wide-sense stationary operators with uncorrelated scattering}, or WSSUS. The function $C_\steta(t,\nu)$ is then called \emph{scattering function} \cite{BelloChar, VanTrees, OPZ}. Our results do not presuppose the stationarity of $\H$, instead, they include it as an interesting special case.
\begin{figure}[ht]

\centering
\def\myscale{0.33}
\subfloat[Tensor square]{ 
%#0 \tikz[scale=\myscale] \tensorsupport; 
\includegraphics{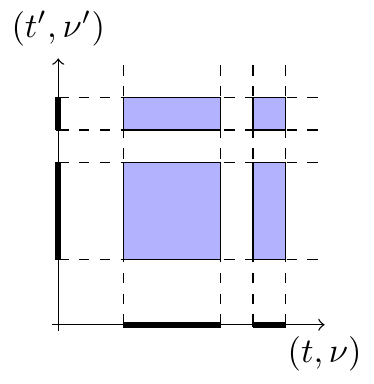}
} 
\subfloat[Arbitrary]{ 
%#1 \tikz[scale=\myscale] \curvysupport; 
\includegraphics{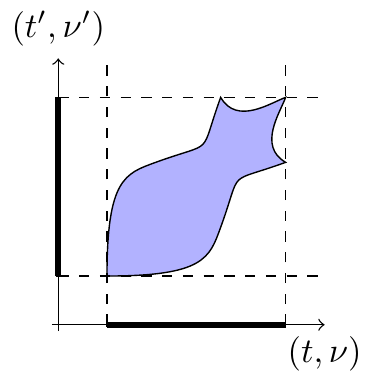}
}
\subfloat[WSSUS]{ 
%#2 \tikz[scale=\myscale] \wssussupport; 
\includegraphics{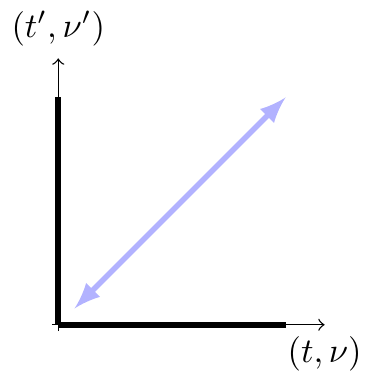}
} 
\caption{Types of distributional support sets of autocorrelations of spreading functions.}
\end{figure}

Under the a priori assumption that the operator $H\colon X\to Y$ belongs to class of linear operators $\Hcal$, identification of the operator from the received echo $H f$ is possible only if for some sounding signal $f$ the mapping $\Phi_f\colon H \mapsto H f$ is injective, 
\begin{equation*}\label{eq:1-1}
 H_1 f = H_2 f \quad \Rightarrow \quad H_1 = H_2.
\end{equation*}
 For linear mappings between Banach spaces to \emph{identify} $\Hcal$ in a \emph{stable} way, we require $\Phi_f$ and its inverse to be bounded:
 \begin{equation}\label{eq:norm.ineq}
A\norm{H_1-H_2}_{\Hcal} \leq \norm{ H_1 f - H_2 f}_Y \leq  B\norm{H_1 - H_2}_{\Hcal} \quad \text{ for all } H_1, H_2\in \Hcal.
\end{equation}
 In those cases when $\Hcal$ is closed under addition, this is equivalent to 
\begin{equation}\label{eq:id.2}
 A\norm{H}_\Hcal \leq \norm{H f}_Y \leq B \norm{H}_\Hcal  \quad \text{ for all } H \in \Hcal
\end{equation}
for some $A,B>0$. The choice of Banach space norms on $Y$ and $\Hcal$  is thus important.  

In the following, we use the symbol $\lesssim$ to abbreviate inequalities up to a constant, that is, $f(t) \lesssim g(t)$ if there exists $C>0$ such that $f(t) \leq C \, g(t)$ for all $t$. Additionally, $f\asymp g$ means $f\lesssim g$ and $f \geqsim g$. 

One major contribution of this paper is to design $Y$ and $\Hcal$ for the identification inequality \eqref{eq:norm.ineq} to hold for stochastic operators accommodating  generalized functions as input signals and generalized stochastic processes as spreading functions. 

In \autoref{thm:sufficiency} we obtain direct generalization of results \cite{PfaWal, KozPfa, Pfander} that deal with deterministic channel operators.  
Cited results show that delta trains as sounding signals identify a large range of so called \emph{operator Paley-Wiener spaces} from the echo $H f$ as  discussed in \autoref{sec:samp_det_op} below. 
For the case of stochastic operator identification, the reconstruction formulas for the \acorr of the spreading function $\Reta$ given the \acorr of the stochastic channel output $\RHf$ can be found in \cite{PfaZh02}, a paper that also addresses sampling theory connection of stochastic operator identification and practical aspects of the proposed technique. 
This paper provides mathematical backbone of the theory developed. 

The possibility of deterministic operator identification by delta trains is linked  to the invertibility of a submatrix of a finite dimensional Gabor system that corresponds to the geometry of the support set \cite{KozPfa}. 
A finite-dimensional Gabor frame is defined as 
\begin{equation*}\label{eq:Gabor.frame}
G \coloneqq \{ M^l \, T^k \, c \}^{L-1}_{k,l=0,}
\end{equation*}
where the finite-dimensional translation operators $T^k$ and modulation operators $M^l$ operating on a vector $c\in \C^L$ are given by 
\begin{equation*}\label{eq:define.fin.dim.TM}
(T^k c)_p = c_{p-k} \quad \text{ and } \quad (M^l c)_p \coloneqq \epi{l p / L} c_p.
\end{equation*}

We extend the same technique to stochastic operators. 
We find that stochastic operator identification using delta trains is still possible if the finite-dimen\-sion\-al matrix $\GG\eval_\Lambda$ is invertible, where the index set $\Lambda$ is now induced by the geometry of the support set of the autocorrelation of the spreading function $\supp \Reta(\tntn)$. 
We will show that analogous to the deterministic case, from the dimensionality considerations it is evident that the restriction on $\supp \Reta(\tntn)$ to have 4D volume less than the critical volume one is still necessary for delta train identification. 
In addition, we show that this constraint to volume less than one also holds if we replace the delta train by any distribution in the considered herein modulation spaces. 

In \cite{PfaZh01} we developed the reconstruction formulas for the stochastic case and discussed more applied aspects of the technique.  
We discovered that in the case of stochastic operator identification, there exist defective configurations of the said support (which we call \emph{patterns}) that prevent successful identification by delta trains.  
Moreover, this effect persists regardless of level of refinement of the time-frequency grid (that is, of the parameters $a, b$ in \autoref{defn:S.is.rectified}). 

In this paper, we also continue the discussion of defective patterns. It turns out that the two families of defective patterns 
that were determined in \cite{PfaZh02} to define operator classes that cannot be identified by weighted delta trains $\Shah_c$, are, in fact, \emph{globally defective} in a sense that \emph{no} sounding signal $f$ identifies them.  

The paper is structured as follows. 
In \autoref{sec:samp_det_op}, we give the well established definitions and theorems from the theory of deterministic operator identification. 
In \autoref{sec:stoch_mod_spaces} we define the Banach spaces $\StOPW(M)$ of stochastic operators that are natural choices to request identifiability of, and develop the functional analysis tools to successfully deal with them, see, in particular, \autoref{thm:H_extension}, \autoref{thm:sufficiency} and \autoref{thm:necessity}.
In \autoref{sec:sufficiency} and  \autoref{sec:necessity} we describe sufficient and necessary conditions on $M$ for the identification of operators in $\StOPW(M)$.
In \autoref{sec:defective.patterns} we introduce criteria for defective patterns. We determine two families of  defective patterns that give rise to support sets $M$ such that $\StOPW(M)$ is impossible to identify any sounding signal $f$. 
In \autoref{sec:bi-infinite} we give a generalized version of the Theorem 2.1 in \cite{Pfa05}, a result needed to derive theorems in the other sections. 
%%%%%%%%%%%%%%%%%%%%%%%%%%%%%%%%%%%%%%%%%%%%%%%%%%%%%%%%%%%%%%%%%%%%%%%
%%%%%%%%%%%%%%%%%%%%%%%%%%%%%%%%%%%%%%%%%%%%%%%%%%%%%%%%%%%%%%%%%%%%%%%
%%%%%%%%%%%%%%%%%%%%%%%%%%%%%%%%%%%%%%%%%%%%%%%%%%%%%%%%%%%%%%%%%%%%%%%
\subsection{Sampling and identification of deterministic operators}\label{sec:samp_det_op}
Equivalently to \eqref{eq:Hf}, any operator $H$ acting on one-variable signals can be represented in a weak sense by 
\begin{enumerate}[(i)]
\item its \emph{time-varying impulse response} $h(x,t) = \int \eta(t,\nu)\epi{\nu x} \d \nu $, then 
\begin{align*}
 (H f)(x) &= \int h(x,t)\, f(x-t) \d t,\\
\shortintertext{
\item its \emph{kernel} $\kernel(x,y) = h(x,x-y) = \int \eta(x-y,\nu) \epi{\nu x} \d \nu,$ then 
} 
(H f)(x) &= \int \kernel(x,y)\, f(y) \d y,\\
\shortintertext{
\item and its \emph{Kohn-Nirenberg symbol} $\KN(x,\xi) \coloneqq \FT_s\eta(t,\nu)$, then 
}
(H f)(x) &= \int \KN(x,\xi) \, \hat{f}(\xi) \, \epi{x\xi} \d\xi.
\end{align*}
\end{enumerate}
In the last equation, the \emph{symplectic Fourier transform} is given by
\[ \left(\FT_s \eta\right)(x,\xi) = \iint \eta(t,\nu)\, \epi{(\nu x - \xi t)} \d t \d\nu.\]

The spreading function $\eta(t,\nu)$ enjoys a particularly simple relationship to the \emph{short-time Fourier transform} (STFT) $V_\phi \colon \S'(\R^d) \to \S'(\R^d)$, defined as 
\begin{equation}\label{eq:STFT}
V_{\phi} f(t,\nu) \coloneqq \ip{f, M_\nu T_t\, \phi}
\end{equation}
for any window $\phi\in \S(\R^d) $.
Note first that the STFT is symmetric with respect to $f$ and $\phi$ up to a unitary phase factor, that is, 
\begin{equation*}\label{eq:V.sym}
\ip{f, M_\nu T_t \phi} = \empi{\nu t}\ip{\phi, M_{-\nu} T_{-t} f}. 
\end{equation*}
For all $f,\phi\in\S(\R^d) $ we then have the useful equality 
\begin{equation*}\label{eq:eta.vs.V}
\ip{Hf,\phi} = \ip{\eta, V_f \phi}. 
\end{equation*}
The inner product $\ip{\placeholder,\placeholder}$ and the dual pairing $\ip{\placeholder, \placeholder}$ in case of distributions are taken to be conjugate linear in the second component. 

We will say that $H$ belongs to an \emph{operator Paley-Wiener space} if the support of the spreading function $\eta(t,\nu) =  \FT_{s} \KN(x,\xi)$ is contained in a subset $S$ of the $(t,\nu)$-time-frequency plane, usually taken to be compact,
\begin{equation*}
\OPW(S) \coloneqq \Set{H \colon L^2(\R) \to L^2(\R)}{\KN\in L^2(\R^2), \quad \supp \FT_s\KN \subseteq S }.
\end{equation*}
The space $\OPW(S)$ becomes a Banach space given a Hilbert-Schmidt norm $\norm{H}_{\OPW} = \norm{\eta_H}_{L^2(\R^2)}$.
Colloquially, we refer to $H$ as \emph{\enquote{bandlimited}} to $S$. 
Based on ideas from \cite{Kailath, BelloMeas}, first results in operator identification were obtained for deterministic operators with spreading function supported on a rectangle of area one in the time-frequency plane \cite{KozPfa}. 
Later it was proven that identification is possible for a larger class of operators with spreading functions supported on a fixed bounded set of arbitrary shape, not just a rectangle, as long as the area of it is less than 1 \cite{PfaWal}. 
To obtain general identification results, it is necessary to \emph{rectify} the support set $S$ as it is done in the theory of Jordan domains. 
\begin{definition}
\label{defn:S.is.rectified}
The set $S$ is \emph{$(a,b,\Lambda)$-rectified} if it can be covered by $L=\frac{1}{ab}$ translations $\Lambda \coloneqq \{(k_j, l_j)\}_{j=0}^{L-1}$ of the rectangle $\rect \coloneqq [0,a)\times[0,b)$ along the lattice $a\Z\times b\Z$: 
\[ 
S \subset \bigcup_{j=0}^{L-1} \rect+(a k_j, b l_j).
\]
\end{definition}

For compact sets the Lebesgue measure agrees with the outer (respectively, inner) Jordan content denoted henceforth $\vol^{+}$  (respectively, $\vol^{-}$). 
\begin{lemma}
\label{lem:cover_det}
For any compact set $S$ of measure $\mu(S)<1$ there exist $a,b>0$ and $\Lambda$ such that $S$ is $(a,b,\Lambda)$-rectified, $L = \frac{1}{ab}$ is prime, and $S \subseteq [0,T] \times[\OO]$, where $Tb=1$ and $\O a = 1$. %$a<\frac{1}{T}$ and $b<\frac{1}{\O}$. 
\end{lemma}
\begin{proof}
A proof without the requirement that $L$ must be prime can be found in \cite{Folland}. 
It is easy to see that requiring $L$ to be prime does not lose generality and can always be achieved at the expense of further reduction of parameters $a,b$ and by allowing the inclusion of some rectangles that do not intersect $S$. 
\end{proof}

\begin{figure}[ht]
\centering
\includegraphics{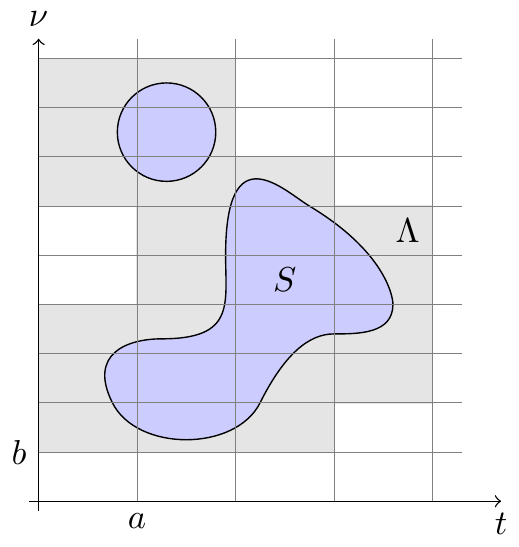}
\caption{An $(a,b,\Lambda)$-rectified support set of a deterministic spreading function in the time-frequency plane\label{fig:rect}. }
\end{figure}

\begin{definition}\label{defn:bupu}
For $a, \eps>0$ we say $r(t)$ generates an \emph{$(a,\eps)$-partition of unity} $\{ r(t+ak) \}_{k\in \Z}$ if
\[ 
\supp r(t) \subseteq (-\eps, a+\eps),\quad \text{ and } \quad \sum\limits_{k\in\Z} r(t+ak) = 1.
\]
An $((a,b),\eps)$-partition of unity of the two-dimensional plane is defined similarly.  If, in addition, the function $r(t)$ belongs to the space $A_c$, then we say that $r(t)$ forms a \emph{bounded uniform partition of unity} \cite{FeiGro92}. Here, $A = \FT L^1(\R)$ is a Fourier image of a space of Lebesgue integrable functions with norm  $\norm{r}_A = \norm{\hat{r}}_{L^1}$, and the subscripted $A_c$ is the subspace of functions in $A$ with compact support.
\end{definition}

\begin{theorem}\cite[Theorem 1.7]{PfaWal}, \cite{PfaWalPreprint} \label{thm:det}
Let $S\subset \RR$ be a compact set with measure $\mu(S)< 1$ such that $S$ is $(a,b,\Lambda)$-rectified in the sense of \autoref{defn:S.is.rectified}, and $S\subseteq [0,\frac{1}{b}] \times [-\frac{1}{2a}, \frac{1}{2a}]$. Then there exists a vector $c\in \C^L$ so that
\[
\Shah_c=\sum_n c_{n \bmod L} \: \delta_{n/L}
\]
identifies $\OPW(S)$ in a sense of \eqref{eq:id.2}, and for any $H\in \OPW(S)$ 
\begin{multline*}
h(x+t,t) = aL\sum_{j=0}^{L-1} \sum_{q\in \Z} a_{j,q} \: H\Shah_c (x-a(k_j+q)) \\
\times r(x-a k_j)\:  \phi(t+a(k_j+q))\epi{b n_j t}
\end{multline*}
unconditionally in $L^2(\R^2)$. 
Here, the coefficients $a_{j,k}$ are uniquely determined by the choice of $\{c_n\}$, and $r(t)$, $\phi(t)$ are functions that are $(a,\eps)$- (resp., $(b,\eps)$-) partitions of unity in time (respectively, frequency) domains, with $\eps>0$ dependent on $S$. 
\end{theorem}
\autoref{thm:det} holds not only if $S$ is compact, but for all regions $S$ whose Jordan outer content is less than one \cite{PfaWal}.

%%%%%%%%%%%%%%%%%%%%%%%%%%%%%%%%%%%%%%%%%%%%%%%%%%%%%%%%%%%%%%%%%%%%%%%
%%%%%%%%%%%%%%%%%%%%%%%%%%%%%%%%%%%%%%%%%%%%%%%%%%%%%%%%%%%%%%%%%%%%%%%
%%%%%%%%%%%%%%%%%%%%%%%%%%%%%%%%%%%%%%%%%%%%%%%%%%%%%%%%%%%%%%%%%%%%%%%
\section{Stochastic modulation and stochastic Paley--Wiener spaces}\label{sec:stoch_mod_spaces}
Let $(\Omega, \mathbb{P}, \Sigma)$ be a probability space, and denote $\RV(\Omega)$ the space of zero-mean complex-valued random variables $r(\w)\colon \Omega \to \C$. 
A general treatment of generalized functions as mappings $\S(\R^d)\to \C$ and of generalized random processes as mappings $\S(\R^d) \to \RV(\Omega)$ dates back to \cite{GelVil}. 
However, for our purpose it is more convenient to deal with Banach spaces as domains and norms instead of Frechet spaces and infinite families of seminorms. We will thus restrict our attention to stochastic versions of the modulation spaces $\Mone$ and $\Minf$.  

\begin{definition}
Let $\phi \in \S(\R^d)$ be a fixed non-zero window, and $1\leq p,q \leq \infty$. Then the modulation spaces $M^p(\R^d)$ consist of all tempered distributions $f\in \S'(\R^d)$ such that the short-time Fourier transform belongs to the Lebesgue space $L^p(\R^{2d})$, that is, 
\[ 
M^p(\R^d) = \Set{f\in \S(\R^d)}{\norm{f}_{M^p(\R^d)} = \paren*{\int\abs{V_\phi f(t,\nu)}^p \d t \d \nu }^{1/p} < \infty}.
\]
$M^p(\R^d)$ is a Banach space; changing $\phi$ leads to equivalent norms. Moreover, $\Minf= (\Mone)'$, where $\Minf$ is defined by replacing the integral by taking the supremum over all $(t,\nu)\in \R^2$. 
\end{definition}
We will use the following property of the space of functions $\Mone(\R^d)$ \cite{FeiGro92}
\[
\Mone(\R^d)\, \widehat{\otimes}\, \Mone(\R^d)=\Mone(\R^{2d}). 
\]
Here, the projective tensor product space $\Mone(\R^d)\, \widehat{\otimes}\, \Mone(\R^d)$ is defined as 
\begin{equation}\label{eq:projective.product}\begin{multlined}
\Mone(\R^d)\, \widehat{\otimes}\, \Mone(\R^d) \coloneqq \Biggl\{ f\in L^1(\R^{2d})\colon f(x,x') \stackrel{L^1}{=} \sum_{n=1}^{\infty} f_n(x) g_n(x') \text{ and }\\
\norm{f}_{\Mone\widehat{\otimes}\Mone} = \inf_{f_n, g_n} \sum_{n=1}^\infty \norm{f_n}_{\Mone} \norm{g_n}_{\Mone} < \infty \Biggr\}.
\end{multlined}\end{equation}

We consider generalized stochastic processes as mappings from $\Mone(\R^d)$ to $\RV(\Omega)$ with autocorrelations in $\Minf(\R^{2d})$, an idea first presented in \cite{FeiHor} and developed further independently of our work in \cite{Wahlberg}. 

We denote by $\Ball{X}$ the unit ball in the normed space $X$. 
\begin{definition}
A \emph{generalized stochastic process} $\f$ on $\R^d$ is a bounded linear map $\Mone(\R^d)\to\RV(\Omega)$. The space of all generalized stochastic processes consists of  the equivalence classes of maps $\Mone(\R^d)\to \RV(\Omega)$ with the usual operator norm 
\[ 
\norm{\f}_{\SMinf(\R^d)} = \sup_{\phi \in \Ball{\Mone(\R^d)}} \norm{\f(\phi)}_{\RV(\Omega).}
\]
It is denoted $\SMinf(\R^d)$; the $\St$ in $\SMinf$ stands for \enquote{stochastic}. 
\end{definition}
\begin{definition}\label{defn:crosscorr}
Let $\f,\g\in \SMinf(\R^d)$. The  \emph{cross-correlation} distribution 
\[ 
R_{\f \g}\colon \Mone(\R^d)\, \widehat{\otimes}\, \Mone(\R^d) \to \C \text{ for all  } \phi,\psi\in\Mone(\R^d).
\]
is defined by 
\[ 
R_{\f \g} (\phi,\conj{\psi}) \coloneqq \ip*{R_{\f \g}, \phi\otimes \conj{\psi}} \coloneqq \ip{\f(\phi), \g(\psi)}_{\RV} = \E{\f(\phi)\, \conj{\g(\psi)}} 
\]
for rank-one tensors $\phi\otimes \conj{\psi}$ and extended in a linear way to finite sums. 
Consider $\theta = \sum_{i=1}^{n} \phi_i \otimes \psi_i$, 
\begin{multline*}	
\abs{R_{\f \g}(\theta)}
= \abs{R_{\f \g} (\sum_{i=1}^{n} \phi_i \otimes \psi_i)} 
=\abs*{\sum_{i=1}^{n} \ip{\f(\phi_i), \g(\psi_i)}_{\RV} } \\
 \leq \sum_{i=1}^{n} \norm{\f(\phi_i)}_{\RV} \norm{\g(\psi_i)}_{\RV} 
\leq \norm{\f}_{\SMinf} \norm{\g}_{\SMinf} \sum_{i=1}^{n} \norm{\phi_i}_{\Mone} \norm{\psi_i}_{\Mone.}
\end{multline*}
Since this holds true for any tensor representation of $\theta$, it follows that 
\[
\abs*{R_{\f \g}(\theta)} \leq  \norm{\f}_{\SMinf}\norm{\g}_{\SMinf} \norm{\theta}_{\Moneone}
\]
for all such $\theta$. We then extend $R_{\f\g}$ to the whole of $\Mone(\R^d) \widehat{\otimes} \Mone(\R^d) = \Mone(\R^{2d})$ by continuity. 
The cross-correlation $\Rf = R_{\f \f}$ of a generalized stochastic process $\f$ with itself is its \emph{autocorrelation}. 
\end{definition}

\begin{lemma}\label{lem:sm}
For all $\f\in \SMinf(\R^d)$, 
$\norm{\f}_{\SMinf(\R^d)}^2 \asymp \norm{\Rf}_{\Minf(\R^{2d}).}$
\end{lemma}
\begin{proof}
We begin with an observation that taking $\psi=\phi$, we guarantee
\begin{multline*}
\norm{\f}^2_{\SMinf(\R^d)} = \sup_{\phi \in \Ball{\Mone(\R^d)}}  \ip{\f(\phi), \f(\phi)}_{\RV} \leq \sup_{\phi, \psi \in \Ball{\Mone(\R^d)}} \ip{\f(\phi), \f(\psi)}_{\RV}  \\
\leq \sup_{\phi, \psi \in \Ball{\Mone(\R^d)}} \norm{\f(\phi)}_{\RV} \norm{\f(\psi)}_{\RV} = \norm{\f}^2_{\SMinf(\R^d).}
\end{multline*}

By \cite[Theorem~5.4, p.~21]{FeiGro92} and its proof, every operator $T\colon\Mone\to\Minf$ is in a one-to-one correspondence with a kernel $\kernel_T\in\Minf(\R^{2d})$ such that $\ip{T \phi, \psi} = \ip{\kernel_T, \conj{\phi}\otimes \psi}$ for all $\phi, \psi \in \Mone(\R^d)$, moreover, $\norm{T}_{\Mone\to\Minf}\asymp \norm{\kernel_T}_{\Minf(\R^{2d})}.$ 
In particular, consider an operator $T\colon \Mone\to\Minf$ such that $T \phi = \Rf(\phi, \placeholder)$. Clearly, $\kernel_T = \Rf$. We now compute 
\begin{align*}\label{eq:2d}
\norm{\f}^2_{\SMinf(\R^d)} &= \sup_{\phi,\psi \in \Ball{\Mone(\R^d)}} \abs*{\ip*{\f(\phi), \f(\psi)}}  \\
&= \sup_{\phi, \psi \in \Ball{\Mone(\R^d)}} \abs{\dualp{T \phi, \psi}}  \\
&= \norm{T}_{\Mone \to \Minf} \asymp \norm{\kernel_T}_{\Minf(\R^{2d}} = \norm{\Rf}_{\Minf(\R^{2d}).} \qedhere 
\end{align*}
\end{proof}
%%%%%%%%%%%%%%%%%%%%%%%%%%%%%%%     \STETA extension 
An observation crucial for operator identification is the characterization of the modulation space $\Mone$ as Wiener amalgam space $W(A, \ell^1)$ \cite{FZ}, defined with $p=1$ by 
\begin{equation}\label{eq:Wiener}
W(A,\ell^p) = \SetBig{f \in A_{loc}}{\norm{f}_{W(A,\ell^p)} = \paren[\Big]{\sum_{n\in\Z^d} \norm{f \: T_n \psi}^p_{A} }^\frac{1}{p} < \infty},
\end{equation}
where $\psi \in A_c(\R^d)$ is a bounded uniform partition of unity, see \autoref{defn:bupu}. The norm $\norm{\placeholder}_{W(A, \ell^p)}$ does not depend (up to equivalence) on the choice of $\psi$. The space $W(A, \ell^\infty)$ is defined similarly, with usual modifications. 

The (distributional) support of a generalized stochastic process $\f$ is defined as the complement of the largest set $S$ such that for any $\phi\in \Mone(\R^d)$ entirely supported within $S$, $\f(\phi)=0$ as elements of $\RV$. 

\begin{theorem}\label{thm:steta_extension}
Any bounded linear operator $\steta \colon M^1(\R^2)\to \RV$ with $\supp \steta$ compact extends to a bounded linear operator $\steta \colon W(A,\ell^\infty)\to \RV$. 
\end{theorem}
\begin{proof}
Without loss of generality, we can rescale the argument variables $(t,\nu)$ in such a way that $\supp \steta \subseteq [-\frac12,\frac12)\times[-\frac12,\frac12)$. 
Let $\psi\in A_c$ be such that it generates a \bupu as in \eqref{eq:Wiener}, and whose support is large enough to cover the support of $\steta$, that is, 
\begin{equation}\label{eq:bupu}
\sum_{n\in \Z^d} \Trans_{n} \psi \equiv 1, \quad \text{ and } \quad \psi\big\vert_{\supp \steta} \equiv 1.
\end{equation}

For any $\phi\in W(A, \ell^\infty)$ by definition we have 
\[ \norm{\phi}_{W(A, \ell^\infty)} \coloneqq \sup_{n\in\Z^d} \norm{\Trans_n \psi \: \phi}_A < \infty. \]
In particular, $\psi\,\phi\in A$, and $\psi\,\phi$ is compactly supported, thus $\psi\,\phi\in W(A,\ell^1)$.  In fact,  since support of $\psi$ is compact, the set $\mathcal{F} \coloneqq \Set{n}{\supp \Trans_n\psi \cap \supp \psi \neq \varnothing}$ is finite,  $\abs{\mathcal{F}} < \infty$. We further observe
\begin{equation}\label{eq:psiphi}\begin{aligned}
\norm{\psi\:\phi}_{W(A, \ell^1)} &= \sumZ{n} \norm{\Trans_n \psi\: \psi \, \phi}_A 
= \sum_{n\in \mathcal{F}} \norm{\Trans_n \psi\:\phi}_A \norm{\psi}_A \\
&\leq \abs{\mathcal{F}} \, \norm{\psi}_A \, \sup_{n\in \mathcal{F}} \norm{\Trans_n \psi\:\phi}_A 
\leq  \abs{\mathcal{F}} \, \norm{\psi}_A \, \norm{\phi}_{W(A,\ell^\infty),}
\end{aligned}\end{equation}
where we used the fact that $A$ is a Banach algebra under pointwise products.
We can now define 
\[ \steta(\phi)\coloneqq \steta(\psi\,\phi).\]
This obviously defines a linear mapping $W(A,\ell^1) \to \RV$. For any choices $\psi, \psi'$ satisfying \eqref{eq:bupu} we have $(\psi-\psi')\phi\equiv 0$ on $\supp \steta$, therefore, it is well defined. Boundedness follows by observing 
\begin{align*} 
\norm{\steta(\phi)}_{\RV} &= \norm{\steta(\psi\:\phi)}_{\RV}  \\
&\leq \norm{\steta}_{W(A,\ell^1)\to \RV} \norm{\psi\:\phi}_{W(A,\ell^1)}  \\
&\leq C \norm{\steta}_{\SMinf(\R^{2d})} \norm{\phi}_{W(A,\ell^\infty),}
\end{align*} 
where we used \eqref{eq:psiphi} and recalled that $\Mone = W(A,\ell^1)$, and $\SMinf = \mathcal{L}(\Mone,\RV)$.
\end{proof}
%%%%%%%%%%%%%%%%%%%%%%%%%%%%%%%     KERNEL (S_0)x(S_0)

\begin{theorem}\label{thm:kernel}
For all $\H \colon \Mone(\R)\to \SMinf(\R)$, there exists a unique spreading function $\steta \colon \Mone(\R^2) \to \RV$ such that for all $f, \phi\in \Mone(\R)$
\[
\dualp{\H f,\phi} = \dualp{\steta, V_f \phi}.
\]
\end{theorem}
\begin{proof}
The proof follows closely the proof of the kernel theorem for modulation spaces in \cite[Theorem~14.4.1, p.~314]{Gro}.
\end{proof}

\begin{definition}
We say that an operator $\H \colon \Minf(\R)\to \SMinf(\R)$ belongs to the stochastic Paley-Wiener class $\StOPW(M)$ if  
\[ 
\dualp{\H f,\phi} = \dualp{\steta, V_f \phi} \quad \text{ for all } f \in \Minf(\R), \phi \in \Mone(\R),
\]
the support of $\steta$ is compact and the distributional support of $\Reta$ is a subset of the closed set $M \subseteq \supp \steta \otimes \supp \steta$. 
The space $\StOPW(M)$ inherits the operator norm $\norm{H}_{\Minf \to \SMinf.}$ 
\end{definition}
%%%%%%%%%%%%%%%%%%%%%%%%%%%%%%%     H extension 
\begin{theorem}\label{thm:H_extension}
Any bounded linear operator $\H \colon \Mone(\R) \to \SMinf(\R)$ with $\supp \steta$ compact extends to a bounded linear operator $\H \colon \Minf(\R)\to \SMinf(\R)$.  Moreover, 
\[ 
\norm{\H f}_{\SMinf} \leq C \norm{\steta}_{\SMinf} \norm{f}_{\Minf} \quad \text{ for all } f\in\Minf(\R).
\]
\end{theorem}
\begin{proof} 
Let $\steta \in \SMinf(\R^{2d})$ with $\supp \steta\subseteq M$ such that $M$ is compact. 
From \cite[Proposition~4.2,~part~1]{PfaWalFeich} we know that for $\phi\in\Mone(\R), f\in \Minf(\R)$, we have $V_{\phi} f \in W(A,\ell^\infty)$ and 
\[ 
\norm{V_f \phi}_{W(A,\ell^\infty)} \leq \norm{\phi}_{\Mone} \norm{f}_{\Minf.} 
\]
By \autoref{thm:kernel}, $\dualp{\H f,\phi} = \dualp{\steta, V_f \phi}$,
and by \autoref{thm:steta_extension}, $\steta$ accepts $V_f \phi \in W(A,\ell^\infty)$ as input, so we can estimate 
\begin{align*} 
 \abs{\dualp{\H f,\phi}} &=\abs{\dualp{\steta, V_f \phi}} \\
&\leq C \norm{\steta}_{\SMinf}  \norm{V_f \phi}_{W(A,\ell^\infty)} \\
&\leq C \norm{\steta}_{\SMinf} \norm{\phi}_{\Mone}\norm{f}_{\Minf.}
\end{align*}
Taking the supremum over all $\phi\in\Ball{\Mone(\R)}$ on both sides gives
\begin{equation*}\label{eq:upper}
\norm{\H f}_{\SMinf} \leq C\norm{\steta}_{\SMinf}\norm{f}_{\Minf,}
\end{equation*}  
and ultimately, for the operator norm,
\begin{equation*}
\norm{\H}_{\StOPW(M)} \leq C\norm{\steta}_{\SMinf.} \qedhere
\end{equation*}
\end{proof}

\begin{lemma}\label{lem:EV=VE}
For all $\f,\g \in \SMinf(\R^d)$, $\displaystyle \ip{V_\phi \f, V_\psi \g}_{\RV} = \mathcal{I}V_{\phi \otimes \conj{\psi}} R_{\f\g},$ where $\mathcal{I}$ is the reflection in the fourth variable: $\mathcal{I}f(\cdot,\cdot, \cdot, \nu') = f(\cdot,\cdot, \cdot, -\nu')$. 
\end{lemma}
\begin{proof}
By \autoref{defn:crosscorr} of $R_{\f\g}$,
\begin{align*}	
\ip{V_\phi \f, V_\psi \g}_{\RV} &= \E{V_\phi \f (t,\nu) \,\conj{V_\psi \g (t',\nu')}}  \\
&= \E{\dualp{\f, M_\nu T_t \phi}\,\conj{\dualp{\g, M_{\nu'}T_{t'} \,\psi}}}  \\
&= \dualp*{R_{\f\g}, M_\nu T_t \phi \,\otimes \,\conj{M_{\nu'}T_{t'} \,\psi}} \\
&= V_{\phi \otimes \, \conj{\psi}} \ R_{\f\g}(t,\nu, t',-\nu').  \qedhere
\end{align*} 
\end{proof}

\begin{remark}
Note that we only used that the space $\RV(\Omega)$ has a Hilbert space structure. The theory developed above still holds if we replace $\RV(\Omega)$ with an arbitrary Hilbert space $\Hcal$, and $\E{\placeholder,\placeholder}$ with the  inner product $\ip{\placeholder, \placeholder}_{\Hcal.}$
\end{remark}
%%%%%%%%%%%%%%%%%%%%%%%%%%%%%%%%%%%%%%%%%%%%%%%%%%%%%%%%%%%%%%%%%%%%%%%
%%%%%%%%%%%%%%%%%%%%%%%%%%%%%%%%%%%%%%%%%%%%%%%%%%%%%%%%%%%%%%%%%%%%%%%
%%%%%%%%%%%%%%%%%%%%%%%%%%%%%%%%%%%%%%%%%%%%%%%%%%%%%%%%%%%%%%%%%%%%%%%
\section{A sufficient condition for the identifiability of \texorpdfstring{$\StOPW(M)$}{StOPW(M)}}\label{sec:sufficiency} 
By identifying the stochastic operator $\H$, we mean determining $\Reta(\tntn)$, the autocorrelation  of the spreading function $\steta_{\scriptstyle \H}(t,\nu)$, from $\RHf(x,x')$, the autocorrelation of the received response $\H f(x)$ to a fixed sounding signal $f$. 
More generally, given a class of stochastic operators $\Hcal$, not necessarily closed under addition, we would like to be able to tell apart any two operators using \eqref{eq:norm.ineq}. 
\newcommand{\Phist}{\Phi_f^{\text{st}}}
Since the autocorrelations in question satisfy $\ip*{\RHf, \phi \otimes \conj{\psi}} = \ip{\Reta, V_{\phi\otimes\conj{\psi}} \: f \otimes \conj{f}}$,
we are trying to invert in a stable way the linear mapping $\Phist\colon \Minf(\R^4) \to \Minf(\R^2)$ given formally and weakly by 
\begin{align}\label{eq:Phist}
\ip*{\Phist (\placeholder), \phi \otimes \conj{\psi}} = \ip{\placeholder, V_{\phi\otimes\conj{\psi}} \: f \otimes \conj{f}}.
\end{align}
We reprise here the definition \eqref{eq:norm.ineq} tailored to the map $\Phist$ and the class $\StOPW(M)$. 
\begin{definition}\label{defn:id.3}
We say that $f \in \Minf(\R)$ \emph{identifies} the set $\StOPW(M)$ if the linear mapping  $\Phist$ given by \eqref{eq:Phist} 
is boundedly invertible, that is, if there exist constants $0 < A \leq  B  < \infty$ such that for all pairs $H_1 , H_2 \in \StOPW(M)$, 
\begin{equation}\label{eq:id.3}
A\norm{R_{\steta_1}-R_{\steta_2}}_{\Minf(\R^4)} \leq \norm{R_{\H_1 f}  - R_{\H_2 f}}_{\Minf(\R^2)} \leq  B\norm{R_{\steta_1} - R_{\steta_2}}_{\Minf(\R^4)}. 
\end{equation}
\end{definition}
Our methods require the support set of the autocorrelations, or rather, their superset $M$, to be rectified satisfying  the symmetry properties of the autocorrelation function.
\begin{definition}\label{defn:M.is.rectified}
We say the set $M$ is \emph{symmetrically $(a,b,\Lambda)$-rectified} if it can be covered by the translations 
$\Lambda \coloneqq \braces{(k_j, l_j, k'_j, l'_j) }_{j=1}^{J}$
of the prototype parallelepiped $\rect^2 \coloneqq [0,a)\times[0,b) \times [0,a) \times [0,b)$ along the lattice $(a,b,a,b)\Z^4$, that is,  
\begin{equation}\label{eq:M.is.rectified}
M \subset \bigcup_{j=0}^{J} \rect^2+ (a k_j, b l_j, a k_j', b l_j'),
\end{equation}
such that the 4D volume of $\rect^2$ is small: $ab=\frac{1}{L}$, with $L$ is prime, and the index set $\Lambda$ is an \emph{admissible} set, in a sense that 
\begin{equation}\label{eq:spd}
(\klkl) \in \Lambda \quad \implies (k,l,k,l), (k'\!,l'\!,k'\!,l'), (k'\!,l'\!,k,l) \in \Lambda.
\end{equation}
The rectification $(a,b,\Lambda)$ is \emph{\label{lbl:precise} precise} if we have equality in \eqref{eq:M.is.rectified}. 
We will always denote $\Gamma \subset \Z^2$ the projection of $\Lambda$ onto the first two (and by symmetry, the second two) indices, that is,  $\Gamma = \{ \gamma = (k,l) \mid (k,l,k,l) \in \Lambda \}$. Clearly, $\Lambda \subset \Gamma \times \Gamma$. 
\end{definition}

The following theorem is an identification result corresponding to operator sampling results in \cite{PfaZh02}. It provides justification to the formal calculations that lead to the reconstruction formula (10) in \cite{PfaZh02}. 
\begin{theorem}
\label{cor:Reta_less_RShah} \label{thm:sufficiency}
Let $M$ be symmetrically $(a,b,\Lambda)$-rectified with $L^2=\abs{\Lambda}$. If some $c\in\C^L$ generates $G(\,:\,,kL+l) = [ M^l \, T^k c ]_{k,l=0}^{L-1}$ such that the submatrix $(\GG)\vert_\Lambda$ is invertible, then 
\[
\norm{\steta}_{\SMinf(\R^2)} \lesssim \norm{\H\Shah}_{\SMinf(\R)} \quad \text{ for any } \H\in\StOPW(M),
\]
where the sounding signal is given by $\Shah_c(x) = \sum_{k\in\Z} c_{k \bmod L}\: \delta(x-a k).$
\end{theorem}
\begin{proof}
In the following, by abuse of notation, $\norm{f}_{\Minf([0,a))} \coloneqq \norm{V_g f}_{L^\infty([0,a)\times \R).}$
Consider 
\allowdisplaybreaks
\begin{align*}
\MoveEqLeft \norm{\H \Shah_c}^2_{\SMinf(\R)} = \norm{R_{\H \Shah_c}(x,x')}_{\Minf(\R^2)}  \\
&= \sup_{n,n'\in \Z} \norm{ R_{\H\Shah_c}(x+an, x'+an')}_{\Minf([0,a)^2)} \\
&= \sup_{j,j'\in \L} \sup_{m,m'\in \Z} \norm{ R_{\H\Shah_c}(x+aj + amL, x'+aj'+am'L)}_{\Minf} \\
&\asymp \sup_{j,j'\in\L} \norm[\Big]{\sumZ{m,m'} R_{\H\Shah_c}(x+an, x'+an') \: M_{-an} \psi(\gamma) \: \conj{M_{-an'} \psi(\gamma')}}_{\Minf,}
\end{align*}
where we use $n=mL+j, n'=m'L + j'$ and apply \autoref{lem:Riesz} below with $\Psi_{n,n'}(\gamma,\gamma') \coloneqq  M_{-an} \psi(\gamma)\: \conj{M_{-an'} \psi(\gamma')}$ with $\psi(\gamma)$  chosen later. 
Consider the expression inside the norms
\begin{multline*}
R_{j,j'}  = \EXP \left( \sumZ{m} \H\Shah_c(x+an)\empi{an\gamma} \psi(\gamma) \right) \\
\times  \left( \sumZ{m'}\H\Shah_c(x'+an')\empi{an'\gamma'} \psi(\gamma') \right)^*.
\end{multline*}
The treatment of both factors is identical, we only detail the first. %It is notable that this can be seen to be the Zak transform of the echo, 
\begin{align*}
R_j &=\sumZ{m} \H\Shah_c(x+an)\:\empi{an\gamma}\\
 &=\sumZ{m}\sumZ{k}c_{n+k}\h(x+a(j+mL), x-ak) \empi{a(j+mL)\gamma} \\
&= \sumZ{m} \sumZ{k} c_{k+mL+j} \: \h(x+aj+amL, x-ak) \empi{amL\gamma} \empi{aj\gamma} \\
&=  \sumZ{k} c_{k+j} \left[ \sumZ{m} \: \h(x+aj+amL, x-ak) \empi{amL\gamma}\right] \empi{aj\gamma} \\
\intertext{--- by \autoref{eq:h.to.steta} ---}
&= b^d \sumZ{k} c_{k+j} \sumZ{l} \: \steta(x-ak, \gamma-bl)\epi{(x+aj)(\gamma-bl)} \empi{aj\gamma}  \\
&= b^d \sumZ{k} \sumZ{l} c_{k+j} \: \steta(x-ak, \gamma-bl)\epi{(x\gamma + blx - jl/L)}  \\
&= b^d \sumZ{k} \sumZ{l} \left( M^{-l} T^{-k} c \right)_j \: \steta(x-ak, \gamma-bl)\epi{x(\gamma-bl)} \\
&= b^d \sumZ{k,l} G_{j,kl} \:  \steta(x+ak, \gamma+bl) \epi{x(\gamma-bl)},
\end{align*}
where in the last step we rename $k=-k, l=-l$ and denote $G_{j,(k,l)} = \left( M^l T^k c \right)_j$.
Now we return to the original $\norm{\H\Shah_c}_{\SMinf(\R),}$ 
\newcommand*{\mysup}{ \sup_{\vphantom{j'\in\L}\nu,\nu' \in \R}  \sup_{\vphantom{j'\in\L}t,t'\in [0,a)} \sup_{\vphantom{j'\in\L}\xi,\xi' \in \R} \sup_{\vphantom{j'\in\L}\tau,\tau'\in [0,b)}  \sup_{\vphantom{j'\in\L}j,j'\in\L}}
\newcommand*{\mystrut}{\vphantom{[\xi'}} 
\begin{align*}
\MoveEqLeft \norm{\H\Shah_c}^2_{\SMinf(\R)} =\sup_{j,j'\in\L} \norm*{R_{j,j'}}_{\: M^\infty([0,a)^2)}\\
&= b^{2d} \sup_{j,j'\in\L} \Big\Vert \EXP \sumZ{k,l} G_{j,kl} \:  \steta(x+ak, \gamma+bl) \epi{x(\gamma-bl)} \: \psi(\gamma)\\
&\qquad \times \Big( \sumZ{k'\!,l'} G_{j',k'l'} \:  \steta(x'+ak', \gamma'+bl') \epi{x'(\gamma'-bl')} \: \psi(\gamma')\Big)^* \Big\Vert_{\Minf([0,a)^2)} \\
&\asymp  \sup_{j,j'\in\L}\Big\Vert \Big\langle \sumZ{\klkl} G_{j,kl} \: \:\conj{G_{j',k'l'} } \: \Reta(x+ak, \gamma+bl, x'+ak', \gamma'+bl') \\
&\qquad \times  \epi{x(\gamma-bl)} \: \psi(\gamma) \:\conj{ \epi{x'(\gamma'-bl')} \: \psi(\gamma')}, \\
&\qquad \times \: M_\nu T_t \: \phi(x) \:\conj{ M_{\nu'} T_{t'} \: \phi(x')} \: M_\xi T_\tau \: \phi(\gamma) \:\conj{ M_{\xi'} T_{\tau'} \: \phi(\gamma')} \Big\rangle \Big\Vert_{L^\infty(([0,a)\times \R)^2),} \\ 
\intertext{where $\phi(x)\in\S(\R)$ is a smooth window as in \eqref{eq:STFT}. We substitute  $x=x+a k, \gamma=\gamma+b l, x'=x'+a k', \gamma' = \gamma'+b l'$ }
&\asymp  \sup_{j,j'\in\L} \Big\Vert \sum_{\klkl\in \Z} \Big\langle G_{j,kl}\:\conj{G_{j',k'l'} }\: \Reta(x,\gamma, x',\gamma') \epi{(x\gamma-x'\gamma')} \\
&\qquad \times \: \psi(\gamma-bl) \:\conj{ \psi(\gamma'-bl')}, \quad M_\nu T_{t+ak} \: \phi(x) \:\conj{ M_{\nu'} T_{t'+ak'} \: \phi(x')} \: \\
&\qquad \times \:  M_{\xi-ak} T_{\tau+bl} \: \phi(\gamma) \:\conj{ M_{\xi'-ak'} T_{\tau'+bl'} \: \phi(\gamma')} \Big\rangle \Big\Vert_{L^\infty} \\
\intertext{denote for simplicity $\widetilde{\Reta}(x,\gamma, x',\gamma') \coloneqq \Reta(x,\gamma, x',\gamma')\epi{(x\gamma-x'\gamma')}$ and narrow the domain of $\tau, \tau'$ to $[0,b)$ } 
&\geqsim \mysup \Big\vert \Big\langle  \sum_{\klkl\in\Z} G_{j,kl}\:\conj{G_{j',k'l'} }\: \widetilde{\Reta}(x,\gamma, x',\gamma')  \\ 
 &\qquad \times \: \psi(\gamma-bl) \:\conj{ \psi(\gamma'-bl')},  \quad  M_\nu T_{t+ak} \: \phi(x) \:\conj{ M_{\nu'} T_{t'+ak'} \: \phi(x')} \\
&\qquad \times   \: M_{\xi-ak} T_{\tau+bl} \: \phi(\gamma) \:\conj{ M_{\xi'-ak'} T_{\tau'+bl'} \: \phi(\gamma')} \Big\rangle \Big\vert \\
\intertext{by selecting $\psi(\gamma)$ such that $\psi(\gamma)\equiv 1$ on $[0,b) + \supp \: \phi(\gamma)$, $\psi(\gamma)$ can be dropped }
&\asymp  \mysup \Big\vert \Big\langle \sum_{\klkl\in\Z} G_{j,kl}\:\conj{G_{j',k'l'}}\:\widetilde{\Reta}(x,\gamma, x',\gamma'),  \\
&\qquad \times \: M_\nu T_{t+ak} \: \phi(x) \:\conj{ M_{\nu'} T_{t'+ak'} \: \phi(x')} \\
&\qquad \times \: M_{\xi-ak} T_{\tau+bl} \: \phi(\gamma) \:\conj{ M_{\xi'-ak'} T_{\tau'+bl'} \: \phi(\gamma')} \Big\rangle \Big\vert \\
\intertext{by the condition on the support of $\Reta$, by \autoref{lem:trim} we can reduce the summation to $(\klkl)\in \Lambda$ }
&\asymp  \mysup \Big\vert \sum_{\klkl\in\Lambda} G_{j,kl}\:\conj{G_{j',k'l'}} \\ %\:\widetilde{\Reta}(x,\gamma, x',\gamma'),  \\
&\qquad \times V_\phi \widetilde{\Reta} (t+ak, \nu, \tau+bl, \xi-ak; t'+ak', \nu', \tau'+bl', \xi'-ak') \Big\vert\\
\intertext{by assumption, the submatrix $\left(\GG\right)\eval_\Lambda$ is invertible, thus the finite vector norms are equivalent  $\norm{(\GG)\eval_\Lambda v}_\infty \asymp \norm{v}_\infty$, and we can continue }
&\asymp  \sup_{\nu,\nu' \in \R \mystrut}  \sup_{t,t'\in [0,a) \mystrut} \sup_{\xi,\xi' \in \R \mystrut} \sup_{\tau,\tau'\in [0,b)\mystrut} \sup_{(\klkl) \in\Lambda\mystrut} \\
&\qquad \Big\vert  V_\phi \widetilde{\Reta} (t+ak, \nu, \tau+bl, \xi-ak; t'+ak', \nu', \tau'+bl', \xi'-ak') \Big\vert\\
\intertext{we can return to $\klkl\in\Z$, and exchange the order of supremums }
&\asymp \sup_{\klkl\in\Z} \sup \Big\{ \abs*{V_\phi \widetilde{\Reta} (t, \nu, \tau, \xi; \:   t', \nu', \tau', \xi')} \text{ where } \nu,\nu' \in \R, \\
& \quad  t,\xi\in a[k, k+1), t',\xi'\in a[k', k'+1), \tau \in b[l, l+1), \tau'\in b[l', l'+1) \Big\} \\
&= \norm{V_\phi\Reta}_{L^\infty}  =  \norm{\Reta}_{M^\infty} = \norm{\steta}_{\SMinf}^2. \qedhere
\end{align*}
\end{proof}

We have used the following terminology and lemmas. 
\begin{definition}\label{def:rieszbasisforspan}
A sequence $\{f_n\}$ in a Banach space $X$ is an $\ell^p$-\emph{Riesz basis for its closed linear
span} if there exist constants $A, B>0$ such that
for every finite sequence $\{c_n\} \in \ell^p(\Z)$
\[
A\,\norm{\{c_n\}}_p \leq \norm{\sum_n c_n\,f_n}_{X} \leq B \norm{\{c_n\}}_p. 
\]
\end{definition}

\begin{lemma}\label{lem:Riesz}
Given an $\ell^\infty$-Riesz system $\{\Psi_{mL + j,m'L+j'}\}_{m,m'\in\Z}^{j,j' \in \L}$ of $\Minf(\R^2)$, 
\begin{multline*}
 \sup_{j,j'\in\L} \sup_{m,m'\in\Z} \abs*{\{ a_{m,m'}\} } \asymp \\
 \sup_{j,j' \in \L} \norm[\Big]{\sumZ{m,m'} a_{m,m'} \: V_{\phi\otimes \conj{\phi}} \: \Psi_{mL+j, m'L+j'} (\txtx) }_{L^\infty(\R^4).}
\end{multline*}
\end{lemma}
\begin{proof}
For any $j,j'\in \L$, 
$
\{ \Psi_{mL+j, m'L+j'} \}_{m,m'\in\Z}
$
are $\ell^\infty$-Riesz bases for their respective linear spans with the same Riesz bounds, that is, 
\begin{align*}
\sup_{m,m'\in\Z} \abs*{\{ a_{m,m'}\} } &= \norm*{\{a_{m,m'}\}}_{\ell^\infty[m,m']} \\
&\asymp \norm*{\sum\nolimits_{m,m'\in \Z} a_{m,m'} \: \Psi_{mL+j, m'L+j'}}_{\Minf(\R^2)} \\
&\asymp \norm*{\sum\nolimits_{m,m'\in \Z} a_{m,m'} \: V_{\phi \otimes \conj{\phi}} \: \Psi_{mL+j, m'L+j'}}_{L^\infty(\R^4).} \qedhere
\end{align*}
\end{proof}

\begin{lemma}\label{eq:h.to.steta}
With the notation used above, 
\[ 
\sumZ{m} \h(x+aj+amL, x-ak) \empi{amL\gamma} = b^d \sumZ{l} \steta(x-ak, \gamma - bl)\epi{(x+aj)(\gamma-bl)}.
\]
\end{lemma}
\begin{proof}
Let the \emph{non-normalized Zak transform}  $\Zak_{\alpha} \colon L^2(\R) \to L^2\left([0, \alpha)\times[0,\alpha^{-1})\right)$ be defined  by
\begin{equation*}\label{eq:defn.Zak}
\Zak_{\alpha}\phi(x,\nu) \coloneqq \sumZ{m} \phi(x-\alpha m )\epi{\alpha m \nu}. 
\end{equation*}
Since $\h(\widehat{\placeholder} , t) = \steta(t,\nu)$, by the Poisson summation formula for Zak transforms \cite{Gro}
\begin{equation*}\begin{split}
\MoveEqLeft \sumZ{m} \h(x+aj+amL, x-ak) \empi{amL\gamma} \\
&= \left(\Zak_{aL} \h(\placeholder, x-ak)\right)  (x+aj, \gamma) \\
&= b^d \epi{(x+aj)\gamma} \, \left(\Zak_b \h(\widehat{\placeholder} , x-ak)\right) (\gamma, -x-aj) \\
&= b^d \epi{(x+aj)\gamma} \, \left(\Zak_b \steta(x-ak, \cdot) \right)(\gamma, -x-aj) \\
&= b^d \epi{(x+aj)\gamma} \, \sumZ{l} \steta(x-ak, \gamma - bl)\empi{(x+aj)bl} \\
&= b^d \sumZ{l} \steta(x-ak, \gamma - bl)\epi{(x+aj)(\gamma-bl)}. \qedhere
\end{split}\end{equation*}
\end{proof}

Let $S_\eps$ denote the $\eps$-\nbhd of a set $S\subseteq \R^n$, that is, 
\[ S_\eps = S+\eps \,\Ball{\R^n}. \]

\begin{lemma}\label{lem:trim}
Let $\Lambda$ index the set of $k\in \Z^n$ such that the $\eps$-\nbhd of the support of $R(x)$ is covered by disjoint translations of a parallelepiped $[0,a) = [0,a_1)\times \dotsb \times [0,a_n)$ along the lattice $a\Z^n$
\begin{equation}\label{eq:supp.of.R}
(\supp R(x))_\eps \subseteq a\bigcup\nolimits_{k \in \Lambda} [k, k+1). 
\end{equation}
Then for any $\phi(x)$ supported on an $\eps$-\nbhd of the origin $\eps\,\Ball{\R^n}$, and for any $t\in [0,a)$, we have  
\[ \sum\nolimits_{k\in\Z^n} a_k \ip*{R(x), T_{t+ak} \phi(x)}  = \sum\nolimits_{k\in\Lambda} a_k \ip*{R(x),  T_{t+ak} \phi(x)}. \]
\end{lemma}
\begin{proof}
For all $t\in [0,a)$, 
\begin{equation}\label{eq:supp}
 \supp T_{t+ak}\phi(x) \subseteq \left[ak, a(k+1)\right)_\eps. 
\end{equation}
Taking complements in \eqref{eq:supp.of.R}, we get 
\begin{equation}\label{eq:nothing}
\supp R(x) \cap \left(\bigcup\nolimits_{k \notin \Lambda} [ak, a(k+1))\right)_\eps = \varnothing.
\end{equation}
Then 
\begin{align*}
\MoveEqLeft \sum_{k\in\Z} a_k \ip*{R(x), T_{t+ak} \phi(x)} =  %
\ip[\Big]{R(x), \sum_{k\in\Lambda} a_k T_{t+ak} \phi(x)} +  \ip[\Big]{R(x), \sum_{k\not\in\Lambda} a_k T_{t+ak} \phi(x)}.
\end{align*}
Since by \eqref{eq:supp} the $\supp \sum_{k\not\in\Lambda} a_k T_{t+ak} \phi(x) \subseteq \left( \bigcup_{k\not\in\Lambda} (ak, a(k+1)) \right)_\eps$, which is disjoint from $\supp R(x)$ by \eqref{eq:nothing}, the second summand is zero.
\end{proof}

We will also need the equation \cite[(5)]{PfaZh02} and the following lemma.
\begin{lemma}\label{lem:equivalence}\cite[Lemma 9]{PfaZh02}
Let $\Lambda$ be a fixed finite pattern, $\Lambda \subseteq \{(0,0),\dotsc, (L-1, L-1)\}$ that satisfies \eqref{eq:spd} and let $G \in \C^{L\times K}$. The following are equivalent:
\begin{enumerate}[(i)]
\item For each \nnd $Y\in \C^{K\times K}$, there exists a unique $X\in \C^{L\times L}$ such that 
			\begin{inparaenum}[a)]
			\item $X$ is \nnd,
			\item $\supp X = \Lambda$, and
			\item $Y = GXG^*$. 
			\end{inparaenum}

\item If an hermitian matrix $N\in \C^{L \times L}$ with $\supp N \subseteq \Lambda$ solves the homogeneous equation $0 = GNG^*$, then $N=0$. 

\item The matrix $\GG\eval_\Lambda$ has a left inverse (is full rank). 
\end{enumerate}
\end{lemma}
%%%%%%%%%%%%%%%%%%%%%%%%%%%%%%%%%%%%%%%
%%%%%%%%%%%%%%%%%%%%%%%%%%%%%%%%%%%%%%%
We are now ready to formulate the first main result about the identification of stochastic operators. 
\begin{theorem}[Sufficient conditions for identifiability] \label{thm:id_stoch}
Let $M$ be a symmetrically $(a,b,\Lambda)$-rectified set such that $L=\frac{1}{ab}$ is prime and $\abs{\Lambda} \leq L^2$. If the submatrix 
\[ 
\GG\eval_\Lambda = \brackets*{\conj{M^l  T^k c} \otimes M^{l'}T^{k'} c }_{(\klkl) \in\Lambda} 
\]
is invertible for $c \in \C^L$, then $\Shah_c(x) = \sumZ{k} c_{k \bmod L} \: \delta(x-a k)$ identifies $\StOPW(M)$.
 If the rectification is \hyperref[lbl:precise]{precise}, then invertibility of $\GG\eval_\Lambda$ is also necessary for the identifiability by $\Shah_c$. 
\end{theorem}
\begin{proof}[Proof of sufficiency]
By  \autoref{thm:H_extension}, \autoref{lem:sm},  \autoref{cor:Reta_less_RShah} and the definition of  operator norms, for any $\H\in\StOPW(M)$ we have 
\begin{equation}\label{eq:main_chain}\begin{multlined}
\norm{\H }_{\StOPW} \leqsim \norm{\steta}_{\SMinf} = \norm{\Reta}^{\frac12}_{\Minf(\R^{4d})}  \leqsim \norm{R_{\H \Shah_c}}^{\frac12}_{\Minf(\R^{2d})}  \\
=\norm{\H \Shah_c}_{\SMinf} \leq \norm{\H}_{\StOPW}\norm{\Shah_c}_{\Minf} \asymp \norm{\H}_{\StOPW.}
\end{multlined}\end{equation}
which means that $\StOPW(M)$ is identifiable by $\Shah_c$. 
\end{proof}

\begin{proof}[Proof of the necessity]
Given an $(a,b, \Lambda)$-symmetrically rectified set $M$ such that the corresponding Gabor submatrix $(\GG)\eval_\Lambda$ is not invertible, by \autoref{lem:equivalence} we can find a non-trivial hermitian matrix $N$ supported on $\Lambda$ in the kernel of $(\GG)\eval_\Lambda$. 
It remains to design a pair of stochastic channels $\H, \H'$ with spreading functions $\steta$ and $\steta'$, respectively, that are indistinguishable by $\Shah_c$. 
By mixing equation (5) from \cite{PfaZh02}, we get for all $(t,\nu)\in [0,a)\times [0,b)$
\[
\brackets[\Big]{\quad  \underbrace{\Zak_{a L}(\H \Shah_c)(t+ap, \nu)}_{\mathbf{Z}_p(t,\nu)} \quad  }_{p=0}^{L-1} \eqms G \: \brackets[\Big]{\quad \underbrace{\steta(t+a k, \nu+b l) \epi{b l t}}_{\steta_{k,l}(t,\nu)}\quad }_{k,l=0.}^{L-1}
\]

\renewcommand{\v}{\boldsymbol{v}}
\newcommand{\Zb}{\boldsymbol{Z}}
\newcommand{\mychi}{\chi_\mysmallbox}
Denoting $\Zb(t,\nu) = [ \Zb_p(t,\nu) ]_{p=0}^{L-1}$ and $\v = [\steta_{k,l}]_{k,l=0}^{L-1}$ (similarly, $\Zb', \v'$),
\begin{align*}
\Vec \E{\Zb \Zb^*}(\tntn) &= (\GG)\eval_\Lambda \Vec \E{\v \v^*}(\tntn), \\
\Vec \E{\Zb' (\Zb')^*}(\tntn) &= (\GG)\eval_\Lambda \Vec \E{\v' (\v')^*}(\tntn).
\end{align*}
We let $\stoch{E} = \braces*{\e_{k,l}}_{k,l=1}^L$ an independent collection of $L^2$ random variables in $\RV$, that is, $\E{\e_{k,l} \conj{\e_{k'\!,l'}}} = \delta_{k,k'} \delta_{l,l'} \coloneqq I(\klkl)$, an identity matrix in $\C^{L^2\times\L^2}$. 
Observe that for a sufficiently large $K>0$, both $I$ and $I + N/K$ are positive-definite. 
Therefore, there exist random vectors $\a(k,l), \a'(k,l)\colon \Z_{L^2} \to \RV$ such that $\E{\a \a^*} = I$ and $\E{\a' (\a')^*} = I+N/K$. 
We set
\[ 
\steta(t,\nu) = \sum_{k,l =0}^{L-1} \a_{k,l} \: \mychi(t - a k, \nu - b l) \empi{b l t}, 
\]
and define $\steta'(t,\nu)$ similarly with $\a'$. This implies $\v_{k,l}(t,\nu) = \a_{k,l} \mychi(t, \nu)$. 
Thus 
\[
\E{ \v \v^*}(\tntn)  = \E{ \v' (\v')^*}(\tntn) + \frac{1}{K} N \: \mychi(t,\nu)\:\mychi(t',\nu').
\]
By construction of $N$, it follows that 
\[
\E{\Zb \Zb^*}(\tntn)  = \E{\Zb' \Zb'^*}(\tntn),
\]
Since Zak transform is invertible, $\Zb$ and $\Zb'$ uniquely determine $\H \Shah_c$ and $\H' \Shah_c$. 
Therefore, we have $\H_1 \neq \H_2$ while $R_{\H_1 \Shah_c}(x,x')  = R_{\H_2 \Shah_c}(x,x')$, that is, $\Phist$ is not bounded below, and $\StOPW(M)$ is not identifiable. 
\end{proof} 

We have shown that the identifiability of the class $\StOPW(M)$ with 4D volume of the support set of the \acorr of the spreading function $\steta(t,\nu)$ less than one is determined by the invertibility of the matrix $\GG\eval_\Lambda$. 

Unlike the deterministic case, where $\OPW(M)$ allows identifiability if the 2D area is smaller than one, in the stochastic case the geometry of the support set $M$ plays a nontrivial role. 
In \cite{PfaZh02} we describe the defective support sets $M$ such that for any refinement of the grid (that is, for $a,b$ arbitrarily small) the corresponding matrix $\GG\eval_\Lambda$ is not invertible, even though the volume $\vol^+(M)$ is less than one, thus, the volume requirement is not sufficient to guarantee identifiability by weighted delta trains. In \autoref{sec:defective.patterns} we show that this phenomenon cannot be fixed by replacing $\Shah_c$ with a different tempered distribution. 

But first, we show that the volume requirement is absolutely \emph{necessary} for identifiability of  $\StOPW(M)$. We prove that in the case of $\vol(M)>1$, the class $\StOPW(M)$ is not identifiable with \emph{any} sounding signal $f\in \Minf(\R)$.

%%%%%%%%%%%%%%%%%%%%%%%%%%%%%%%%%%%%%%%%%%%%%%%%%%%%%%%%%%%%%%%%%%%%%%%
%%%%%%%%%%%%%%%%%%%%%%%%%%%%%%%%%%%%%%%%%%%%%%%%%%%%%%%%%%%%%%%%%%%%%%%
%%%%%%%%%%%%%%%%%%%%%%%%%%%%%%%%%%%%%%%%%%%%%%%%%%%%%%%%%%%%%%%%%%%%%%%
\section{A necessary criterion for the identifiability of \texorpdfstring{$\StOPW(M)$}{StOPW(M)}} \label{sec:necessity}
The main goal of this section is to prove that the class $\StOPW(M)$ is not identifiable if $\vol^-(M)>1$ (\autoref{thm:necessity}). 

In fact, we prove a stronger and more technical result, \autoref{thm:Phist.is.MM}, that ties stability of the evaluation operator $\Phist$ to the existence of a bounded left inverse of a certain bi-infinite matrix dependent on $f$ and on the geometry of $M$. 
By \autoref{defn:id.3},  the stability of $\Phist$ is equivalent to identifiability of the class $\StOPW(M)$ by a sounding signal $f$, thus proving \autoref{thm:necessity}, as well as Theorems \ref{thm:k2k4} and \ref{thm:k33} below. 

In the following, we shall replace the set $M$ with a subset $U$ that has volume larger than one, and that has a precise rectification with $\lambda$ having cardinality being a perfect square. As $\StOPW(M)$ contains $\StOPW(U)$, non-identifiability of the latter implies the same for the former.

The proof of \autoref{thm:Phist.is.MM} uses ideas from \cite[Theorem 3.6]{KozPfa} and \cite[Theorem 4.1]{PfaWal}, but it requires also a sophisticated analysis of stochastic modulation spaces. We will need the following basic results from time-frequency analysis, given here without proof. 
\begin{theorem}\label{thm:gaborframeexistence}
\cite{Lyu92,SW92,Sei92b} The Gabor system
$\{M_{a' p}\, T_{b' q} \, g_0 \}_{p,q \in\Z}$ is an $\ell^\infty$-Riesz basis for $\Minf(\R^d)$ if $a',b'>0$ satisfy $a'b'<1$ and  $g_0(x) = e^{-x^2/2}$.
\end{theorem}

Let $\eta_0 \in \S(\R)$ with values in $[0,1]$ and
\[ 
\eta_0(x) =
\begin{cases}
  1, & x \in \left[\frac12 - \frac{1}{2\lambda}, \frac12 + \frac{1}{2\lambda}\right), \\    
  0, & x \notin [0,1].
\end{cases}
\] 
Define $\etarect (t,\nu)=\eta_0(t/a) \otimes\eta_0(\nu/b)$.   Then $\supp \etarect  \subseteq \rect = [0,a)\times [0,b)$. 
\begin{lemma}\label{lem:elementarybound} \cite[Lemma 4.12]{Pfander}
Fix $\lambda>1$ with $1<\lambda^4<\tfrac{\Aleph}{L}$ and choose $\etarect$ as above. Then the operator $P\in \OPW^{\infty}([0,a)\times[0,b))$ with $\etarect$ as its spreading function has the following properties.
\begin{enumerate}[a)]
\item The families 
\begin{align*}
\mathrm{H} &=  \braces{\eta_{\klmn}(t,\nu)}_{\klmn\in\Z} = \braces{M_{\left( \lambda k/a,\lambda l /b\right)} \: T_{\left(a m, b n\right)}\: \etarect(t,\nu)}_{\klmn\in\Z},\\
\Eta &= \braces{\eta_{\klmn}(t,\nu) \otimes \conj{\eta_{\klmnp}(t',\nu')}}, \\
\P &= \{ P_{k,l,m,n} \}_{\klmn \in \Z}  = \braces*{  M_{\lambda k/a}\: T_{a m -\lambda l / b }\: P\: T_{\lambda l /b} \:  M_{bn-\lambda k/a}  }_{\klmn\in\Z}
\end{align*}
are $\ell^\infty$-Riesz bases for their closed linear spans in $\Minf(\R^2)$, $\Minf(\R^4)$ and $\OPW^{\infty}([0,a)\times[0,b))$, respectively. 

\item $P$  is a time-frequency localization operator in the following sense: there exists  a function
$d_0\colon \R\rightarrow\R_0^+$,  which decays rapidly at infinity,
that is, $d_0(x)=\littleo{x^{-r}}$ for all $r\in \N$, and which has the
property that for all $f\in \Minf(\R)$ we have for $x\in \R$ and $\xi \in \widehat{\R}$, 
\[ 
\abs{Pf(x)}\leq \norm{f}_{\Minf(\R)}\,d_0(x), \quad \text{ and }  \quad \abs{\widehat{Pf}(\xi)}\leq \norm{f}_{\Minf(\R)}\,d_0(\xi).
\]
\end{enumerate}
\end{lemma}
%%%%%%%%%%%%%%%%%%%%%%%%%%%%%%%%%%%%%%%%%%%%%%%%%%%%%%%%%%%%%%%%%%%%%%%
%%%%%%%%%%%%%%%%%%%%%%%%%%%%%%%%%%%%%%%%%%%%%%%%%%%%%%%%%%%%%%%%%%%%%%%
%%%%%%%%%%%%%%%%%%%%%%%%%%%%%%%%%%%%%%%%%%%%%%%%%%%%%%%%%%%%%%%%%%%%%%%
\begin{theorem}\label{thm:Phist.is.MM} 
Fix $f \in \Minf(\R)$. 
Let $U$ be a bounded and measurable set in $\R^4$ with a \hyperref[defn:M.is.rectified]{precise $(a,b,\Lambda)$-rectification}, where $L=\frac{1}{ab}$ is prime, and $\abs{\Lambda} = J^2$ for some $J>0$. 
 Let
\begin{equation}\label{eq:tilde.Lambda}
\Ltilde\coloneqq \Set{(\klmn,\klmnp)}{\klkl\in\Z, (\mnmn)\in \Lambda} \subseteq \Z^4\times\Z^4
\end{equation}
 be a cylinder set extruded from $\Lambda$, and form the bi-infinite matrix
\[
(\MM)\eval_{\Ltilde} = \brackets[\Big]{\M_{p,q; \klmn} \: \conj{\M_{p'\!,q'; \klmnp}} }_{\pqpq\in\Z^4}^{(\klmn, \klmnp)\in\Ltilde}
\]
from a subset of columns of the tensor product matrix $\MM$ indexed by $\Ltilde$, where 
\begin{equation*}\label{eq:M}
\M_{ p,q; \klmn} = \ip*{M_{\lambda k / a}\: T_{a m + \lambda l / b  } \: P \: T_{-\lambda l / b} \: M_{ b n- \lambda k / a} \, f, \: M_{\lambda^2 \alpha p} \: T_{\lambda^2 \beta q} \: g_0},
\end{equation*}
with $\lambda, \alpha, \beta>0$ positive constants such that $\lambda^4 < \min(\frac{J}{L}, \frac{1}{\alpha \beta})$, $g_0(x) = e^{-\frac{x^2}{2}}$ a gaussian window, and $P$ is a time-frequency localization operator as defined in \autoref{lem:elementarybound}.

If  the matrix $(\MM)\eval_{\Ltilde}$ is not stable, then the evaluation operator 
\[ 
\Phi^{\text{St}}_f\colon\Minf(\R^4) \to \Minf(\R^2),\quad R_{\steta_{\H}} \mapsto \RHf
\]
defined by \eqref{eq:Phist} is \emph{not} bounded below, and inequality \eqref{eq:id.3} fails to hold.
\end{theorem}
\begin{proof}
We denote $\Gamma$ the projection of $\Lambda$ onto the first two (and by symmetry, the second two) indices, that is,  $\Gamma = \{ \gamma = (m,n) \mid (m,n,m,n) \in \Lambda \}$. Clearly, $\Lambda \subset \Gamma \times \Gamma$, and $\abs{\Gamma} \geq L$. 
\renewcommand*{\d}{\boldsymbol{d}}   
\newcommand{\Xpattern}{Y}

\begin{figure}[ht]
\includegraphics{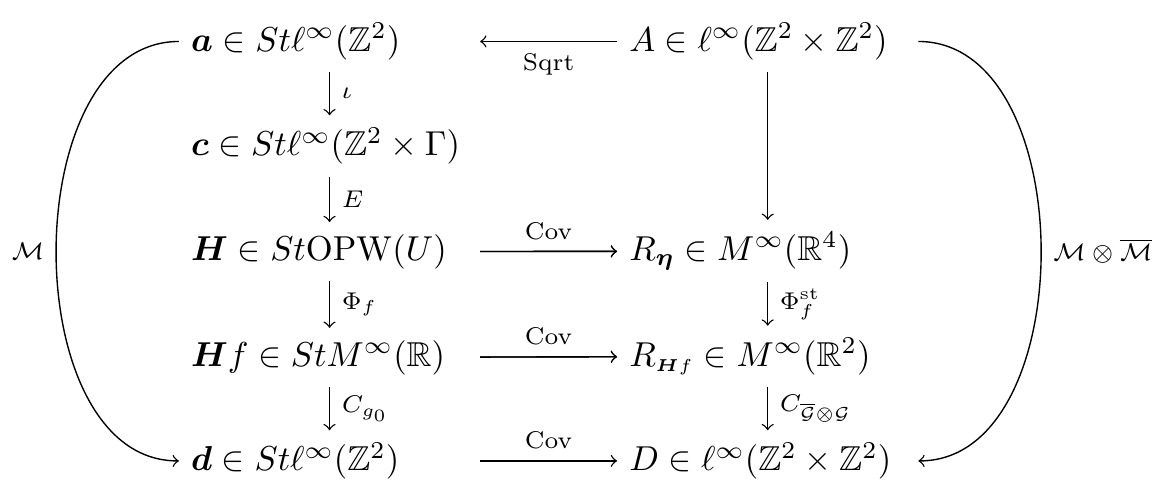}
\caption{The construction diagram.\label{fig:diagram}}
\end{figure}
We follow the construction depicted in \autoref{fig:diagram}, one mapping at a time. 
\begin{enumerate}[(i)]
\item We start with an arbitrary \nnd bi-infinite matrix $A \in \ell^\infty_{+}(\Z^2\times\Z^2)$, and its matrix square root $\B \in \ell^\infty_{+}(\Z^2\times\Z^2)$ \cite{Chui}. 
Let \label{eq:E.stoch} $\stoch{E}$ \label{pg:stoch.E} be a fixed collection of independent random variables that forms a countable basic sequence in $\RV$ indexed by $\Z^2$. Let $\a = \B \stoch{E}$ be a discrete stochastic process in 
\begin{multline*}
\Slinf(\Z^d) = \ell^\infty[\Z^d; \RV] = \\
\SetBig{\a \colon \Z^d \to \RV}{\norm{\a}_{\Slinf(\Z^d)} = \sup_{z\in\Z^d} \norm{\a_z}_{\RV} < \infty}.
\end{multline*}
Similarly to the proof of \autoref{lem:sm}, it is easy to see that $\norm{\a}_{\Slinf(\Z^d)} = \norm{\Ra}_{\ell^\infty(\Z^{2d})}^{\frac12}$, where $\Ra$ stands for a bi-infinite matrix of covariances $\E{\a \a^*}$. 
The map $\Sq\colon A \mapsto \a$ is bounded and bounded below, since 
\[
\norm{\a}^2_{\Slinf(\Z^2)}= \norm{\Ra}_{\infty} =  \norm{\E{\B \stoch{E} \stoch{E}^* \B^*}}_{\infty} = \norm{\B \B^*}_{\infty} = \norm{A}_{\infty.}
\] 
\item As the set $\Gamma$ is finite, we can trivially identify $\Slinf(\Z^2)$ with $\Slinf(\Z^2 \times \Gamma)$. We can now relabel and \enquote{decorrelate} in part the random variables $(\a_{k,l})_\gamma = \a_{k, l\abs{\Gamma}+\gamma}$ to produce the desired \acorr pattern $\Lambda$ in the $\gamma$ variable by setting 
\[ \c_{\klg} = \sum_{i\in\Gamma} Q_{\gamma, i} \, (\a_{\k, l})_i.
\]
Here, $Q\colon\Gamma\times\Gamma\to\C$ is an invertible scalar-valued matrix that satisfies $Q Q^* = \Xpattern$ with $\Xpattern$ such that $\supp \Xpattern = \Lambda$. 
A choice of a positive definite matrix $\Xpattern$ with a given admissible support set $\Lambda$ is always possible by constructing a strongly diagonally dominant matrix and using that $(\gamma, \gamma) \in \Lambda$ for $\gamma \in \Gamma$. 
The covariance of $\c$ then satisfies $\E{\c_{\klg} \, \c_{\klgp}^*} = 0$ for all $(\gamma, \gamma') \notin \Lambda$, that is, it is supported on a cylinder $\Ltilde$.
We have the norm equivalence
\begin{multline*}
\norm{\c}_{\Slinf(\Z^2\times\Gamma)} = \sup_{\kl} \: \sup_{\gamma} \norm{\c_{\kl,\gamma}}_{\RV} \\
= \sup_{\kl} \: \sup_{\gamma} \norm{(Q \, \a_{\kl})_{\gamma}}_{\RV} \asymp \norm{\a}_{\Slinf(\Z^2\times\Gamma),}
\end{multline*}
that is, the mapping $\iota\colon \Slinf(\Z^2) \to \Slinf(\Z^2\times \Gamma), \quad \a \mapsto \c$ is also bounded and bounded below. 

\item We define a synthesis map $E\colon \Slinf(\Z^2\times \Gamma) \to \StOPW(U)$
\begin{equation*}\label{eq:synthesismap}\begin{split}
E(\c) = \badH &= \sum_{k,l\in\Z} \sum_{\gamma \in \Gamma}  \: \c_{\klg} 
\: \underbrace{M_{\lambda k / a}\: T_{a m_\gamma + \lambda l / b  } \: P \: T_{-\lambda l / b} \: M_{ b n_\gamma- \lambda k / a}}_{P_{k,l,m_\gamma,n_\gamma}} 
\end{split}\end{equation*}
where, as in \autoref{lem:elementarybound}, the operator $P_{\klmn}$ has the spreading function $M_{(\lambda  k / a, \lambda l / b ) } \: T_{(a m, b n)}\: \etarect(t,\nu)$. It is easy to verify that the operator $\badH$ has a spreading function
\[
 \steta_E =  \sum_{\kl\in\Z} \sum_{\gamma \in \Gamma}  \: \c_{\klg} \: M_{(\lambda  k / a, \lambda l / b ) } \: T_{(a m_\gamma, b n_\gamma)}\: \etarect(t,\nu).
\]
We abuse the notation slightly by denoting $\gamma = (m_\gamma, n_\gamma)$. 
The map $E\colon \Slinf(\Z^2\times \Gamma) \to \StOPW(U)$ is stable and bounded because 
\begin{align*}
\MoveEqLeft \norm{\badH}_{\StOPW(U)}^2 = \norm{\steta_E(t,\nu)}^2_{\SMinf(\R^2)} \\
&= \norm{R_{\steta_E}(\tntn)}_{\Minf(\R^4)} \\
 &= \Big\lVert \EXP \adjustlimits \sum_{\kl\in\Z} \sum_{\gamma \in \Gamma} \c_{\klg}(\w)  \: 
										M_{(\lambda  k / a, \lambda l / b) } \: T_{(a m_\gamma, bn_\gamma)}\etarect(t,\nu)\\
& \myindent \times \adjustlimits \sum_{\klp\in\Z} \sum_{\gamma' \in \Gamma} \conj{ \c_{\klgp} (\w)\: 
										M_{(\lambda  k' / a, \lambda l' / b) } \: T_{(a m_{\gamma'}, bn_{\gamma'})}\etarect(t',\nu')  }\Big\rVert_{\Minf(\R^4)} \\
&=\Big\lVert\adjustlimits \sum_{\klkl\in\Z} \sum_{\gamma, \gamma' \in \Gamma}  \E{ \c_{\klg} \, \c_{\klgp}^* }  \: M_{(\lambda  k / a, \lambda l / b) } \: T_{(a m_\gamma, bn_\gamma)}\,\etarect(t,\nu)\\
&\myindent \times\: \conj{M_{(\lambda  k' / a, \lambda l' / b) } \: T_{(a m_{\gamma'}, bn_{\gamma'})}\,\etarect(t',\nu')}   \Big\rVert_{\Minf(\R^4)} \\
&= \norm*{\E{ \c_{\klg} \, \c_{\klgp}^* }_{\klkl \in \Z, \gamma, \gamma' \in \Gamma}}_{\ell^\infty(\Z^4\times\Gamma^2)} \\  
&\asymp \norm{\braces*{\c_{\klg}}_{\kl\in\Z, \gamma \in \Gamma}}^2_{\Slinf(\Z^2\times\Gamma),}
\end{align*} 
where we used the fact that $\eta_{\klmn}\otimes \conj{\eta_{\klmnp}}$ is a Riesz basis for its span (by \autoref{lem:elementarybound}).

\item Now, given a stochastic operator $\badH \in \StOPW(U)$, we apply the evaluation functional $\Phi_f \colon \StOPW(U) \to \SMinf(\R)$ to obtain 
\[ 
\Phi_f \badH = \badH f.
\]  

\item By assumption, $\paren*{\lambda^2 \alpha} \paren*{\lambda^2 \beta} < 1$, therefore, the Gabor system 
\[
\Gscr =  \braces*{ g_{pq} = M_{\lambda^2 \alpha p} \: T_{\lambda^2 \beta q} \, g_0 }_{p,q\in\Z}
\]
and hence, the tensor product system 
\[
\GGscr =  \braces*{\gg_{\pqpq}=  \conj{g_{p,q}} \otimes g_{p'\!,q'}}_{\pqpq\in\Z}
\]
are frames by \autoref{thm:gaborframeexistence}. 
Thus, the analysis map (with respect to the frame $\Gscr$ given by
\begin{align*}
  C_{g_0}\colon \SMinf(\R)  \to  \Slinf(\Z^2),\quad  \boldphi \mapsto
         \braces*{ \ip{ \boldphi , g_{p,q} }    }_{p,q\in\Z}
\end{align*}
is bounded and stable, since for all $\d = C_{g_0} \boldphi$ we have 
\begin{multline*}
\norm{\d}_{\Slinf(\Z^2)}^2 = \norm{R_{\d}}_{\ell^\infty(\Z^4)} 
= \sup_{\pqpq\in\Z} \norm{\EXP\ip*{\boldphi, g _{p,q}}\conj{\ip*{\boldphi, g_{p'\!,q'}}}}_{\RV} \\
= \sup_{\pqpq\in\Z} \abs*{\ip*{R_{\boldphi}, \gg_{\pqpq}}}
\asymp \norm{R_{\boldphi}}_{\Minf(\R^2)} 
= \norm{\boldphi}^2_{\SMinf(\R).}
\end{multline*}

\item Let us denote $D= \E{\d \d^*}$ the covariance matrix of $\d = \paren*{C_{g_0} \circ \Phi_f \circ E \circ \iota} \, \a$, and compute 
\begin{equation*}\begin{split}
D_{\pqpq} 
&= \sum_{\substack{\klkl\in\Z \\ (m,n),(m'\!,n')\in\Gamma}} A_{\klmn; \klmnp} \: \underbrace{\ip{ P_{\klmn}\,  f, g_{p,q} }}_{\M_{p,q; \klmn}}  \\
& \qquad \qquad \times \:  \conj{\ip{P_{\klmnp} \, f, g_{p'\!,q'}}}  \: \Xpattern_{\mnmn} \\
&= \sum_{\substack{\klkl\in\Z \\ (\mnmn )\in\Lambda}} A_{\klmn; \klmnp} \: \M_{p,q; \klmn}  \:  \conj{\M_{p'\!,q'; \klmnp}},
\end{split}\end{equation*}
or, in a matrix form, 
\begin{equation}\label{eq:MM}
\Vec D = (\MM)\eval_{\Ltilde}\: \Vec A.
\end{equation}

\item Denote for brevity $\MMM =  (\MM)\eval_{\Ltilde} $ the above map that acts on bi-infinite matrices $\el[4]{\infty}$ (implicitly identified henceforth with infinite vectors $\ell^\infty(\Z^8)$). By assumption, $\MMM$ does not have a bounded left inverse, that is, for any $\eps>0$ there exists a vector $x \in \ell^\infty_0(\Z^8)$ such that $\norm{x}_{\ell^\infty(\Z^8)} = 1$, but $\norm{(\MM)\eval_{\Ltilde} x}_{\ell^\infty(\Z^4)} <\eps$. 
We can now trivially identify a compactly supported vector $x\in \ell^\infty(\Z^8)$ with  $X\in \ell^{\infty}( \Z^4 \times \Z^4)$, a bi-infinite matrix with finitely many nonzero entries such that 
\[ 
\norm{X}_{\el[4]{\infty}} = 1, \text{ and } \norm{\MMM X}_{\ell^\infty(\Z^4)} < \eps. 
\]

\item Due to the tensor product nature of the map $\MM$ and the \hyperref[eq:spd]{symmetry} of the set $\Ltilde$, 
\begin{align*}
\MoveEqLeft \paren*{\MMM X^*}_{\pqpq} \\ 
&= \sum_{\klkl\in\Z, (\mnmn)\in\Lambda} \M_{p,q; \klmn} \conj{\M_{p'\!,q'; \klmnp}} \: \conj{X_{\klmnp; \klmn}} \\
&= \sum_{\klkl\in\Z, (\mnmn)\in\Lambda}\conj{\M_{p'\!,q'; \klmnp}\conj{\M_{p,q; \klmn}}  \: X_{\klmnp; \klmn}} \\
&= \conj{\paren*{(\MM)\eval_{\Ltilde} \Vec (X)}_{p'\!,q'\!,p,q}} \\
&= \conj{(\MMM X)_{p'\!,q'\!,p,q}}. 
\end{align*}
Thus, the conjugate transpose matrix $X^*$ also satisfies $\norm{X^*}_\infty =1$ and $\norm{ \MMM X^*}_\infty<\eps.$
Furthermore, consider hermitian matrices $X_1 = \tfrac{X+X^*}{2}$ and $X_2 = \tfrac{X-X^*}{2i}$. Since $X=X_1+iX_2$, by reverse triangle inequality, $\max(\norm{X_1}_\infty,\norm{X_2}_\infty) \geq \frac12$. Let the maximum be obtained by $X_1$ without loss of generality. By triangle inequality, 
\[ 
\norm{ \MMM X_1}_\infty = \norm{\MMM \, \tfrac{X+X^*}{2}}_\infty \leq \norm{\MMM  X}_\infty \leq \eps. 
\]

Since the square root map $\Sq$, the rearrangement map $\iota$, the synthesis map $E$, the analysis map $C_{g_0}$, and the covariance map $\COV$ are all bounded and bounded below (see the diagram on \autoref{fig:diagram}), as well as the inequality $\norm{\H}_{\StOPW(U)} \lesssim \norm{\steta}_{\SMinf(\R^2)}$ proven in \autoref{thm:H_extension}, there exist constants $K, K'>0$ such that 
\begin{equation*}\begin{aligned}
&\norm{R_{\steta_{\badH}}}_{\Minf(\R^4)} \geq K \norm{A}_{\ell^\infty(\Z^2 \times \Z^2)}\\
&\norm{R_{\badH f}}_{\Minf(\R^2)}\leq K' \norm{D}_{\ell^\infty(\Z^2 \times \Z^2)}. 
\end{aligned}\end{equation*}
We can now select a pair  of \nnd matrices $A_1=(K+1)\Id$ and $A_2 = A_1 + X_1 / (K \norm{X_1}_\infty)$ such that  $\norm{A_1-A_2}_{\ell^\infty(\Z^2 \times \Z^2)} = 1 / K$ and $\norm{\MMM A_1 - \MMM A_2}_{\ell^\infty(\Z^2 \times \Z^2)} \leq \eps.$ 
The corresponding operators $\badH_1, \badH_2 \in \StOPW(U)$ then satisfy 
\[
\norm{R_{\steta_{\badH_1}} - R_{\steta_{\badH_2}}}_{\Minf(\R^4)} \geq  1,\text{ and } \norm{R_{\badH_1 f} - R_{\badH_2 f}}_{\Minf(\R^2)} \leq  \eps' = K\eps/2 \to 0.
\]
\end{enumerate}
 \end{proof}
 
Thus, the identifiability of $\StOPW(U)$ by a fixed signal $f$ is tied to the invertibility of a bi-infinite matrix $(\MM)\eval_{\Ltilde}$, where $\M$ depends on $f$. 
We identify several cases for the geometry of $U$ in which the bi-infinite matrix fails to be bounded below: the case of the excessive volume, $\vol(U)>1$, or the presence of two types of defective patterns. 
Defective patterns are explored in detail in \autoref{sec:defective.patterns}. 

These results follow as corollaries from a general and rather technical \autoref{cor:my.bi-infinite} that guarantees that if we can find a growing sequence of singular minors within $\MMM$ that are dominated by $\lambda$-slanted diagonals, and the matrix $\MMM$ has decay away from the slanted diagonal, then the matrix $\MMM$ does not have a bounded left inverse. 

\begin{definition}\label{defn:decay.away}\cite{ABK, Pfa05}
We say that the entries of a matrix $M\colon \ell^{p_1}(\Z^{d_1}) \to \ell^{p_2}(\Z^{d_2})$ decay away from the $\lambda$-slanted diagonal  if for some positive constants $r_1, r_2, r>0$, 
\begin{equation*}\label{eq:m.decay}
\abs{M_{z_2, z_1}} \leq w(\lambda \norm{z_2}_\infty - \norm{z_1}_\infty) \:  (1+\norm{z_1}_\infty)^{r_1} \: (1+\norm{z_2}_\infty)^{r_2}
\end{equation*}
for some decreasing function $w(x) = \littleo{x^{-r}}, x\in\R$.
\end{definition}
Consistent with the standard linear algebra literature, we call the submatrices 
\begin{align*} 
M_{[N, \lambda N]} \coloneqq \Big[ M_{z_2, z_1 } \colon  \norm*{z_1}_\infty \leq \lambda N, \norm*{z_2}_\infty \leq  N\Big]
\end{align*}
the \emph{$\lambda$-slanted principal minors} of dimensions $(\lambda N, N). $

\begin{figure}[ht]
\centering
\includegraphics{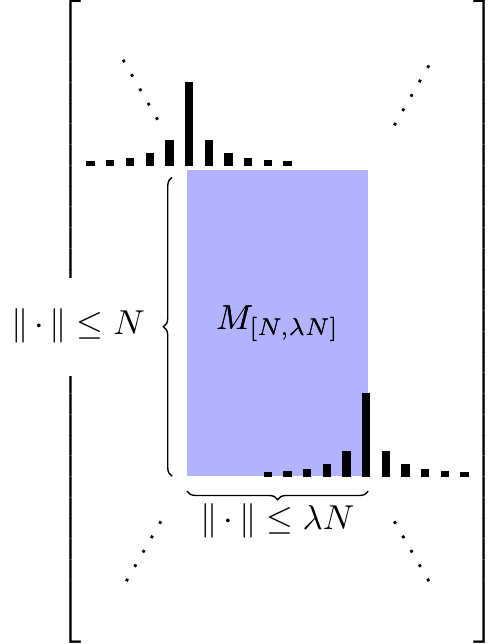}
\caption{A $\lambda$-slanted principal minor. } 
\end{figure}

\begin{theorem}\label{cor:my.bi-infinite}
Let $1\leq p_1, p_2 \leq \infty$, $\frac{1}{p_1} + \frac{1}{q_1} = 1$.  
Let $\MMM = [M_{z_2, z_1}]: \ell^{p_1}(\Z^{d_1}) \to \ell^{p_2}(\Z^{d_2})$ be a complex bi-infinite matrix that has the following properties. 
\begin{enumerate}[(i)]
\item \label{item:simple.ii} the entries of $M$ decay away from the $\lambda$-slanted diagonal, that is, 
\begin{equation*} 
\abs{\MMM_{z_2, z_1}} \leq w(\lambda \norm{z_2}_\infty - \norm{z_1}_\infty) 
\end{equation*}
for some rapidly decreasing function $w(x) = \littleo{x^{-r}}$ for any $r\in \R_+$, 
\item \label{item:simple.iii} and for any $N_0 \in \N$ there exists $N > N_0$ such that the slanted principal central minor
\[ 
\MMM_{[N, \lambda N]}= \brackets[\Big]{M_{z_2, z_1} \colon \size{z_1} \leq \lambda N,  \ \size{z_2} \leq N}
\]
is singular, 
\end{enumerate}
then $\MMM$ has no bounded left inverses. 
In fact, for any $\eps>0$, there exists a vector $x\in\ell^{p_1}_0(\Z^{d_1})$ such that $\norm{x}_{p_1} = 1$ and $\norm{\MMM x}_{p_2}<\eps$. 
\end{theorem}
\begin{proof}
The proof is similar to the proof of \cite[Theorem 2.1]{Pfa05}. We provide it in \autoref{sec:bi-infinite}. 
\end{proof}

We will show that a bi-infinite matrix $\MMM = (\MM)\eval_{\Ltilde}$ defined in \eqref{eq:MM} --- upon rearranging the indices to take advantage of the sparsity of the set $\Ltilde$ --- exhibits the above \hyperref[defn:decay.away]{decay phenomenon}. It is important to note that only in \autoref{thm:necessity}, we need to require $\lambda > 1$.

\begin{lemma}\label{lem:Mtilde.is.unstable}
Let the set $\Lambda$ be an admissible set in the sense of \autoref{eq:spd}. Further, assume that $\abs{\Lambda}=\Aleph^2$, and index the entries of $\Lambda$ by $j,j'=1,\dotsc, \Aleph$, that is, $(\mnmn)_{(j,j')} \in \Lambda$. Upon transformation of variables, namely, 
\[
\sigma=k, \  \tau = l \Aleph +j, \ \sigma' = -k', \  \tau' = l'\Aleph + j', 
\]
let the matrix $\Mtilde\colon \el[2]{\infty} \to \el[2]{\infty}$ be defined as 
\[
\Mtilde_{\pqpq; \sitau,\sitaup} = \M_{p,q,k,l,m_{(j,j')},n_{(j,j')}} \conj{\M_{p'\!,q'\!,k'\!,l'\!,m'_{(j,j')}, n'_{(j,j')}}} \raisebox{-4pt}{,}
\]
where $\M$ is given in \autoref{thm:Phist.is.MM}, with the constants $\alpha, \beta$ are fixed to be $\alpha = \frac{1}{a}$, and $\beta=\frac{1}{\Aleph b}$. 
The matrix $\Mtilde \cong (\MM)\eval_{\Ltilde}$ has  decay away from the $\lambda$-slanted diagonal property from \autoref{defn:decay.away} with $r_1=r_2=0$ and an arbitrary $r\in\N$, that is, 
\[
\abs{\Mtilde_{\pqpq; \sitausitau}} < w(\norm{\SITAU}_\infty - \lambda \norm{\PQ}_\infty),
\]
where $w(x) = \littleo{\abs{x}^{-r}}$ for all $r\in \R$.
\end{lemma}
\begin{proof} 
Here it is crucial that the amount of active boxes in a pattern is a perfect square $\Aleph^2$. This can be easily achieved by refining the rectification and by allowing to increase $\Lambda$. 
The following analysis is the same as in \cite{KozPfa}, only with 2D indices. 

\begin{equation}\label{eq:decay}\begin{split}
\abs*{{\Mtilde}_{\pqpq}^{\sitau,\sitaup}} 
&= \abs{\M_{p,q; \klmnjj}\:\conj{\M_{p'\!,q'; \klmnjjp}}} \\ 
&= \abs*{\ip*{ \paren{P_{\klmnjj} \otimes \conj{P_{\klmnjjp}}} (f\otimes \conj{f}) , \conj{g_{p,q}}\otimes g_{p'\!,q'} }} \\
&= \abs*{\ip*{ M_{\lambda \kk / a} \: T_{a \mm + \lambda \ll / b } \: \PP \: T_{-\lambda \ll / b} \: M_{b \nn - \lambda \kk / a}\: \ff, \gg_{\pp,\qq} }},
\end{split}\end{equation}
where we use upright letters for two-dimensional indices and simple tensor products, that is, $\PP = P\otimes \conj{P}, \ff = f\otimes \conj{f}, \kk = (k,-k')$, $\ll = (l,l'), \mm = (m_\jjp, m'_\jjp), \nn = (n_\jjp,-n'_\jjp), \pp = (p,p')$, and $\qq = (q,q')$. Also, $\jj = (j,j'), \ss = (\sigma, \sigma'), \tt = (\tau,\tau')$.

\newcommand*{\phirm}{\upphi}
By \autoref{lem:elementarybound}, the properties of $\PP$ are such that for any $\phirm(x)\in\Minf(\R^2)$ and any $r\in \N^2$ we have the decay
\begin{equation*} 
 \abs{\PP \phirm(x)} \leq \norm{\phirm}_{\Minf(\R^2)}  d_{0}(x) = \littleo{\abs{x}^{-r}},
\end{equation*}
hence, for some $d_1(x) = \littleo{\abs{x}^{-r}}$, 
\[
T_{a \mm - \lambda \jj / b} \: \abs{\PP \, \phirm(x)} \leq \norm{\phirm}_{\Minf(\R^2)} \: \sup_{\jj=(1,1), \dotsc, (J,J)}  T_{a\mm - \lambda \jj / b} \:d_0(x) = \norm{\phirm}_{\Minf(\R^2)} \: d_1(x).
\]
 In particular, we continue \eqref{eq:decay} with  $\phirm(x)= T_{-\lambda \ll / b } \: M_{b \nn  - \lambda \kk / a} \: \ff$ and $\tt=\ll \Aleph+ \jj$, noting that $\norm{\phirm}_{\Minf(\R^2)} = \norm{\ff}_{\Minf(\R^2)}$, 
\newcommand*{\dd}[1]{\paren*{d_{#1} \otimes \conj{d_{#1} }}}
\begin{align*}
\abs*{\Mtilde_{\pqpq}^{\sitau,\sitaup}}& =  \abs*{\ip*{ M_{\lambda \kk / a} \: T_{a \mm + \lambda \ll / b } \: \PP \: T_{-\lambda \ll / b} \: M_{b \nn - \lambda \kk / a}\: \ff, M_{\lambda^2 \alpha \pp} \: T_{\lambda^2 \beta \qq} \: \gg_{0,0}} }\\
&\leq \ip*{  T_{  \lambda  \ll / b } T_{a \mm}  \abs*{\PP  \phirm  },  T_{\lambda^2\beta \qq} \: \gg_{0,0} }  \\
& = \ip*{  T_{  \lambda   (\ll \Aleph + \jj) / (b\Aleph) } T_{a \mm - \lambda \jj / b }  \abs*{\PP  \phirm  },  T_{\lambda^2\beta \qq} \: \gg_{0,0} }  \\
& \lesssim  \norm{\ff}_{\Minf(\R^2)}  \ip*{  T_{  \lambda  \tt / (b\Aleph) } \dd{1},  T_{\lambda^2\beta \qq} \: \gg_{0,0} }  \\
& = \norm{\ff}_{\Minf(\R^2)} \ip*{  T_{  \frac{\lambda}{b\Aleph}  \paren*{\tt  - \lambda \beta b \Aleph \qq} } \dd{1},  \: \gg_{0,0} }  \\
&\leq \norm{\ff}_{\Minf(\R^2)} w_1 \paren*{\tau - \lambda \beta  b\Aleph q} \: w_1 \paren*{\tau' - \lambda \beta  b\Aleph q'}, 
\end{align*}
where $w_1(x) = \ip*{T_{\frac{\lambda}{b \Aleph} x} d_1, \: g_0} = \littleo{\abs{x}^{-r}}$. 

Similarly, by taking the Fourier transform on both sides of the inner product,
\begin{align*}
 \Mtilde_{\pqpq}^{\sitau,\sitaup} & = \abs*{\ip*{ M_{\lambda \kk / a} \: T_{a \mm + \lambda \ll / b } \: \PP \: T_{-\lambda \ll / b} \: M_{b \nn - \lambda \kk / a}\: \ff, 
					M_{\lambda^2 \alpha \pp} \: T_{\lambda^2 \beta \qq} \, \gg_{0,0}} }\\
 &= \abs*{\ip*{T_{\lambda \kk / a} \: M_{-a \mm - \lambda \ll / b } \: \FT \: \PP \: T_{-\lambda \ll / b} \: M_{b \nn - \lambda \kk / a}\, \ff, T_{\lambda^2 \alpha \pp} \: \abs*{ M_{\lambda^2 \beta \qq} \: \widehat{\gg_{0,0}}}}} \\
 &\leq \ip*{T_{\frac{\lambda}{a}(\ss - \lambda \alpha \pp)} \abs{\FT \: \PP \, \phirm}, \widehat{\gg_{0,0}}} \\
 &\lesssim \norm{\ff}_{\Minf(\R^2)} \ip*{T_{\frac{\lambda}{a}(\ss - \lambda \alpha a \pp)} \dd{0}, \: \widehat{\gg_{0,0}}}\\
 & \leq \norm{\ff}_{\Minf(\R^2)} w_2\paren*{\sigma - \lambda \alpha a p} \:  w_2\paren*{\sigma' - \lambda \alpha a p'}.
\end{align*}
where $w_2(x) = \ip*{T_{\frac{\lambda}{a} y} d_0, \, \widehat{g_0}}$, and we justify our choice of $\alpha$ and $\beta$ by observing $b \beta \Aleph = 1$ and $a \alpha = 1$.
The proof is complete by setting 
\[
w(x) = \max \braces*{ \norm{w_1}_{\infty} w_1(x),  \: \norm{w_2}_{\infty} w_2(x)} = \littleo{x^{-r}} \quad \text{ for all } r\in\N. \qedhere
\]
\end{proof}

 \renewcommand{\d}{\:\mathrm{d}}

We are now ready to complete this section by proving the main theorem. 
\begin{theorem}\label{thm:necessity}
If an admissible bounded measurable set $M\subseteq \R^4$ has $\vol^{-}(M)>1$, the class $\StOPW(M)$ is not identifiable by any $f\in\Minf(\R)$.
\end{theorem}
\begin{proof}
By the theory of Jordan domains, for any such $M$ we can find a subset $U$ of $M$ with a precise symmetric $(a,b,\Lambda)$-rectification $U$ such that $\vol (U) >1$, and at the expense of further refinement of the lattice, we can guarantee a parallelepiped count to be a perfect square, $\abs{\Lambda} = \Aleph^2$. Consider that the volume of $U$ is precisely $(ab)^2 (\Aleph^2) = (\Aleph / L)^2 > 1$. Therefore, we can choose $1 < \lambda < \sqrt[4]{\frac{\Aleph}{L}}$ such that selecting the constants $\alpha= \frac{1}{a}, \beta = \frac{1}{b \Aleph}$ satisfies the requirement of \autoref{thm:Phist.is.MM} necessary for the system $\Gscr$ to be a frame. 
\[ 
\lambda^4 \alpha \beta = \lambda^4 \frac{1}{a b J} = \lambda^4 \frac{L}{J} < 1.
\]
Thus, the conditions of \autoref{lem:Mtilde.is.unstable} are fulfilled, and the there defined matrix $\Mtilde$ decays away from the $\lambda$-slanted diagonal. 

It remains to observe that whenever $\lambda > 1$, we can always find a $1< \pambda< \lambda$ such that there exists a growing sequence of $\pambda$-slanted principal minors $\Mtilde_{[N,\pambda N]} \colon \C^{16\floor{\pambda N}^4} \to \C^{16N^4}$ of the matrix $\Mtilde$ that are singular, since for $N$ sufficiently large, $\floor{\pambda N} > N$. 
Therefore, the \autoref{cor:my.bi-infinite} holds, the matrix $\Mtilde$ is not stable, and neither is the evaluation operator $\Phist$. This proves that the set $\StOPW(U)$, and hence, its superset $\StOPW(M)$ is not identifiable by any $f$.
\end{proof}

  %%%%%%%%%%%%%%%%%%%%%%%%%%%%%%%%%%%%%%%
 %%%%%%%%%%%%%%%%%%%%%%%%%%%%%%%%%%%%%%%
 %%%%%%%%%%%%%%%%%%%%%%%%%%%%%%%%%%%%%%%
\renewcommand{\d}{\boldsymbol{d}}
\renewcommand*{\c}{\boldsymbol{c}}   
\newcommand{\Gtilde}[1]{I_{#1}} 

\section{Defective sets}\label{sec:defective.patterns}
In \cite{PfaZh02}, we have described the geometrical conditions on the rectification pattern $\Lambda$ of $M$ that prevented the identifiability by delta train, based on the condition in \autoref{cor:Reta_less_RShah} that the matrix $\GG\eval_{\Lambda}$ must be invertible. 
It turns out that the state of affairs is not the deficiency of delta train identifiers $\Shah_c$, but rather a consequence of the tensor product structure of the autocorrelation support. 
In this section we show that the sets that we determined to be unidentifiable by weighted delta trains, remain so even if we allow $f$ to be an arbitrary distributional sounding signal. We suggest to call such patterns \emph{globally defective}. 

In \cite[Definition 10]{PfaZh02}, we defined a pattern to be defective it a weighted delta train could not identify the corresponding operator classes. We discovered two families of patterns that were defective in that sense. We extend the definition of the second family here and show that these patterns are indeed globally defective. 
\begin{definition}\label{defn:defective.families}
Let $\Gamma \subset \Z^2$ be a set such that $\abs{\Gamma}  = L^2$. 
Let $\Lambda \subset \Gamma \times \Gamma$ be an \hyperref[eq:spd]{admissible set} in $\Z^4$. 
Let the graph $G$ on the vertices indexed by $\Gamma$ be such that its adjacency matrix is the indicator matrix $\Id_{\Lambda}$.
If for some sets $\Gamma_1, \Gamma_2 \subset \Gamma$ such that $\Gamma_1 \cap \Gamma_2 = \varnothing$, and $\abs{\Gamma_1} + \abs{\Gamma_2} > L$, 

\begin{enumerate}[(i)]
\item the graph $G$ contains two disjoint complete subgraphs $K_{\Gamma_1}$ and $K_{\Gamma_2}$ (that is $K_{\Gamma_i}, i=1,2$ is a complete graph on the vertices in $\Gamma_i$), then we say $\Lambda$ contains a \emph{two squares} pattern on $\Gamma_1, \Gamma_2$. 

\item the graph $G$ contains a complete bipartite subgraph $K_{\Gamma_1, \Gamma_2}$, then we say  that $\Lambda$ contains a \emph{butterfly} pattern on $\Gamma_1, \Gamma_2$. 
\end{enumerate}
We give examples of each type of pattern in \autoref{fig:defective.patterns}.
\end{definition}

\begin{figure}[ht]
\centering
\def\myscale{0.45}
\subfloat[$K_2 \sqcup K_4$, \enquote{two squares}] 
{ \label{fig:defectiveK2K4} 
\includegraphics{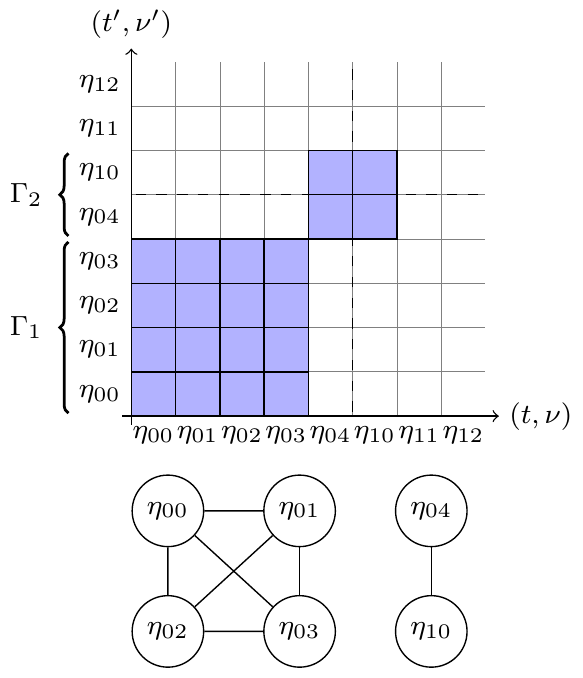}
}
\subfloat[$K_{3,3}$, \enquote{butterfly}] 
{ \label{fig:defectiveK33} 
\includegraphics{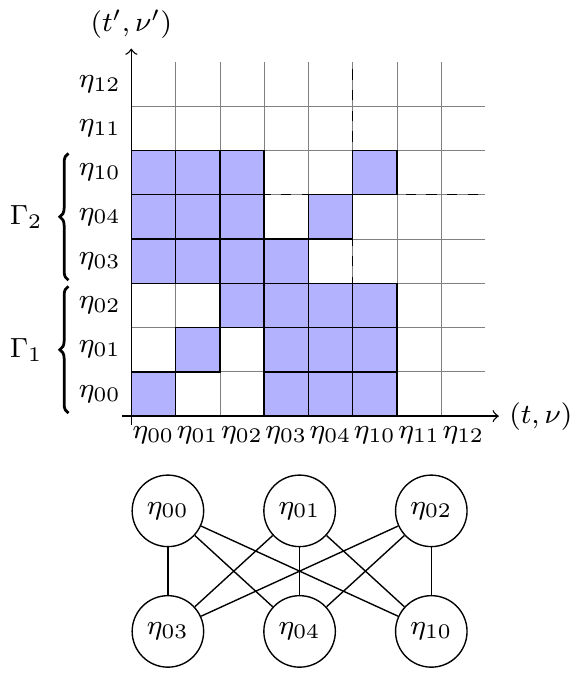}
}
\caption{Defective patterns, $L=5$. \label{fig:defective.patterns}}
\end{figure}

\begin{theorem}\label{thm:k2k4}
Let $M \subset \R^4$  have a precise symmetric $(a,b,\Lambda)$-rectification such that $\Lambda$ contains a two squares pattern on $\Gamma_1, \Gamma_2$ as defined in \autoref{defn:defective.families}, that is, $\abs{\Gamma_1} + \abs{\Gamma_2} \geq L+1, \Gamma_1 \cap \Gamma_2 = \varnothing$, and $(\Gamma_1 \times \Gamma_2) \cup (\Gamma_2 \times \Gamma_1) \subset \Lambda$. 
Assume also that $\abs{\Lambda} = \Aleph^2$ for some $\Aleph \in \N$.

The class $\StOPW(M)$ is not identifiable by any $f\in \Minf(\R)$, that is, for any $\eps>0$, there exists an operator $\H\in \StOPW(M)$ such that $\norm{\H f}_{\SMinf(\R)} <\eps \norm{\H}_{\StOPW(M).}$
\end{theorem}
\begin{proof}
The case $\Aleph > L$ is covered by the \autoref{thm:necessity}, so  we can assume that $\lambda^4 < \Aleph/L \leq 1$, $\lambda$ to be chosen later. The proof will follow the same lines as the proof of \autoref{thm:necessity}.

The system $\Gscr = \braces*{ g_{p,q} = M_{\lambda^2 \alpha p} T_{\lambda^2 \beta q} g_0}$ is a frame, if we choose $\alpha= \frac{1}{a}, \beta = \frac{1}{b \Aleph}$ so that 
\[ 
\lambda^4 \alpha \beta = \lambda^4 \frac{1}{a b J} = \lambda^4 \frac{L}{J} < 1.
\]
By \autoref{lem:Mtilde.is.unstable}, the there defined matrix $\Mtilde$ has decay away from the $\lambda$-slanted diagonal.
In order to apply \autoref{cor:my.bi-infinite} and  \autoref{thm:Phist.is.MM}, we need to show that for some $\pambda < \lambda$ there exists a growing sequence of singular $\pambda$-slanted principal minors $\Mtilde_{[N, \pambda N]}$ of the matrix $\Mtilde$. 
However, here, due to $\lambda<1$, these submatrices are tall and skinny rather than short and fat, as in the case of \hyperref[thm:necessity]{excessive volume}, so a closer look is necessary. 

For a fixed $N$ and $\pambda<\lambda$ to be chosen later, let $M_{[N, \pambda N]}$ be a $\pambda$-slanted principal minor
\begin{multline*}
\Mtilde_{[N, \pambda N]} = \Big[\Mtilde_{\pqpq; \sitausitau} \colon \\
\norm*{\PQ}_\infty \leq N, \quad \norm*{\SITAU}_\infty \leq \floor{\pambda N}\Big]. 
\end{multline*}
Consider the variable transformation 
\[
\sigma=k, \  \tau = l \Aleph +j, \ \sigma' = -k', \  \tau' = l'\Aleph + j', \ (\mnmn) = (m_j, n_j, m'_{j'}, n'_{j'}) \in \Lambda. 
\]
We are guaranteed to have  $\norm*{\SITAU}_\infty \leq \floor{\pambda N}$ if  
\[ 
\abs{k},\abs{k'} \leq \floor{\pambda N}, \quad \abs{l},\abs{l'} \leq \floor*{\frac{\floor{\pambda N} - \Aleph}{\Aleph}} \coloneqq N'.
\]

\begin{figure}[ht]
\centering
\includegraphics{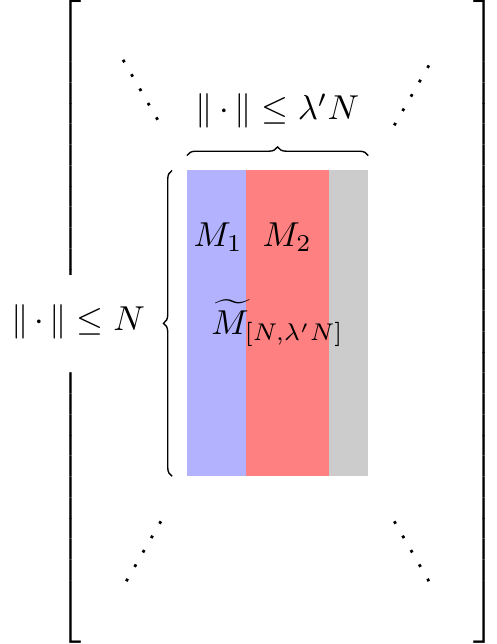}
\caption{The $\pambda$-slanted principal minor $M_{[N, \pambda N]}$ and its column submatrices $\Em{1}$, $\Em{2}$. }
\end{figure}
Let us denote $\Em{1}, \Em{2}, \Em{\Lambda}$ and $\Em{\text{full}}$ the submatrices comprising columns of $\Mtilde_{[N, \pambda N]}$ such that 
$\abs{k},\abs{k'} \leq \floor{\pambda N}, \abs{l},\abs{l'} \leq  N'$ and $(\mnmn) \in \Gamma_1 \times \Gamma_1$ (respectively, $\Gamma_2 \times \Gamma_2$, $\Lambda$, and $(\Gamma_1 \cup \Gamma_2) \times (\Gamma_1 \cup \Gamma_2)$).  
We will also denote the corresponding index sets $\Gtilde{1}, \Gtilde{2}, \Gtilde{\Lambda}$ and $\Gtilde{\text{full}},$ for example, 
\begin{equation}\label{eq:Em.full}
\Gtilde{\Lambda} =\begin{multlined}[t] \Big\{ (\sitausitau)\in \Z^4  \colon \\
\abs{k},\abs{k'} \leq \floor{\pambda N},  \abs{l},\abs{l'} \leq  N', (\mnmn) \in \Lambda\Big\}.
\end{multlined}
\end{equation}
We will show that the submatrix $\Em{\Lambda}$ does not have full rank.

Since $(\Gamma_1 \times \Gamma_1) \cap (\Gamma_2 \times \Gamma_2)  = \varnothing$, the submatrices $\Em{1}$ and $\Em{2}$ do not overlap.
We will show, however, that their respective ranges intersect. Observe that $\Em{i} =  \M_i \otimes \conj{\M_i}$, $i=1,2$, where 
\begin{multline}\label{eq:Mu}
\M_i = \brackets[\Big]{\M_{p,q; \klmn} \colon \abs{p}, \abs{q} \leq N,  \abs{k}\leq \floor{\pambda N},  \abs{l} \leq N', (m,n) \in \Gamma_i}.
\end{multline}
The column spans of $\M_1$ and $\M_2$ in a $(2 N +1)^2$-dimensional ambient space have dimensions $(2\floor{\pambda N}+1)(2N'+1)\abs{\Gamma_i}$, $i=1,2$. For $N$ sufficiently large, they must have a nonempty intersection, because 
\begin{multline*}
(2\floor{\pambda N}+1)(2N'+1)(\abs{\Gamma_1} + \abs{\Gamma_2}) 
\geq (2\pambda N - 1) \paren*{\frac{2(\pambda N-1)}{J} - 1} (L+1)  \\
= 4(\pambda)^2  \frac{L+1}{J} N^2 + \bigO{N} 
>  (2N+1)^2
\end{multline*}
whenever $(\pambda)^2 > \frac{\Aleph}{L+1}$. We can always find two distinct real numbers $\lambda$ and $\pambda$ such that 
\begin{equation}\label{ineq:lambda.pambda}
\frac{\Aleph}{L+1} < (\pambda)^2 < \lambda^2 < \sqrt{\frac{\Aleph}{L}},
\end{equation}
for $\Aleph<L$, for example, $\pambda = \sqrt{\frac{\Aleph}{L+\frac12}}, \lambda = \sqrt{\frac{\Aleph}{L}}$. 
Therefore, there exist nonzero vectors 
$b_1 \in \Dom(\M_1)$ and $b_2 \in \Dom(\M_2)$  such that $\M_1 b_1 + \M_2 b_2 = 0$. 
It follows that for the vector $b$ given by
\[
b = 
% \begin{bmatrix}
% b_1 \otimes \conj{b_1} \ \, {}\\ b_2 \otimes \conj{b_2}  \ \,{}
% \end{bmatrix}
% \begin{array}{c}
% {\}} \ I_1 \\
% {\}} \ I_2 
% \end{array}
\raisebox{-3ex}{  %\input{bb.tikz}
\includegraphics{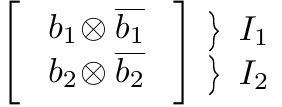}
}
\]
we have $[\Em{1} \mid \Em{2}]\:  b = 0$. 
Therefore, padding $b$ with zeros if necessary, we have found a non-trivial vector in the kernel of $\Mtilde_{[N, \pambda N]}$. By \eqref{ineq:lambda.pambda}, the condition \eqref{item:simple.iii} of \autoref{cor:my.bi-infinite} is satisfied, and the bi-infinite matrix $\Mtilde$ does not have a left inverse. Hence, the evaluation operator $\Phist$ is not stable, and the class $\StOPW(M)$ is not identifiable by $f$. 
\end{proof}

The case of a butterfly defective pattern is proven similarly. 
\begin{theorem}\label{thm:k33}
Let $M \subset \R^4$ have a precise symmetric  $(a,b,\Lambda)$-rectification such that $\Lambda$ contains a butterfly pattern on $\Gamma_1, \Gamma_2$ as defined in \autoref{defn:defective.families},  that is, $\abs{\Gamma_1} + \abs{\Gamma_2} \geq L+1, \Gamma_1 \cap \Gamma_2 = \varnothing$, and $(\Gamma_1 \times \Gamma_2) \cup (\Gamma_2 \times \Gamma_2) \subset \Lambda$. 
Assume also that $\abs{\Lambda} = \Aleph^2$ for some $\Aleph \in \N$.

The class $\StOPW(M)$ is not stochastically identifiable by any $f\in \Minf(\R)$, that is, for any $\eps>0$, there exists a operator $\H\in \StOPW(M)$ such that $\norm{\H f}_{\SMinf(\R)} <\eps \norm{\H}_{\StOPW(M).}$
\end{theorem}
\begin{proof}
Following the same construction as above, with the same $\lambda, \alpha$ and $\beta$, consider the same matrices $\Em{\text{full}}$ defined in \eqref{eq:Em.full}, and $\M_1, \M_2$, defined in \eqref{eq:Mu}, and the vectors $b_1 \in \Dom \M_1$ and $b_2 \in \Dom \M_2$ that satisfy $\M_1 b_1 + \M_2 b_2 =0$. 
The matrix $\Em{\text{full}}$ can easily be seen to be a tensor product $ \begin{bmatrix} \M_1  &  \M_2 \end{bmatrix} \otimes \conj{\begin{bmatrix} \M_1  &  \M_2 \end{bmatrix}}$. 
Consider a matrix supported on $\Gtilde{\Lambda} \subset \Gtilde{\text{full}}:$
\[ 
B = 
\raisebox{-6ex}{
\includegraphics{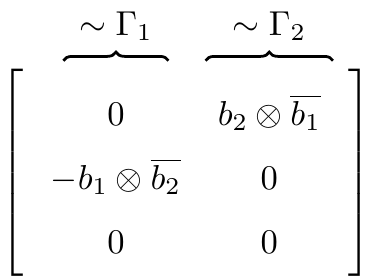}
}
% \input{bbbb.tikz}
% \begin{bmatrix}
%  0 & b_2 \otimes \conj{b_1}  \\ -b_1 \otimes\conj{b_2} & 0 \\  \underbrace{0}_{~\Gamma_1}   & \underbrace{0}_{~\Gamma_2}    \\
% \end{bmatrix}
\]
Then 
\begin{align*}
\Em{\text{full}} \: \Vec B 
&= \begin{bmatrix} \M_1  &  \M_2 \end{bmatrix} \otimes \conj{\begin{bmatrix} \M_1  &  \M_2 \end{bmatrix} }  \Vec B \\
&= \M_2 b_2 \otimes \conj{ \M_1 b_1} - \M_1 b_1 \otimes  \conj{\M_2 b_2} \\
&= 0. 
\end{align*}
Thus, the matrix $\Em{\text{full}}$ has a non-trivial vector $\Vec B$ in its kernel. Since $B$ is supported on $\Gtilde{\Lambda}$, the submatrix $\Em{\Lambda}$ is also singular, and hence, so is $\Mtilde_{[N, \pambda N]}$. 
By \autoref{cor:my.bi-infinite}, the bi-infinite matrix $\Mtilde$ does not have a left inverse. 
By \autoref{thm:Phist.is.MM}, the evaluation operator $\Phist$ is not stable, and the class $\StOPW(M)$ is not identifiable by $f$. 
\end{proof}

\section{Invertibility of bi-infinite matrices}\label{sec:bi-infinite}
\renewcommand{\d}{\:\mathrm{d}}

We generalize the invertibility \cite[Theorem 2.1]{Pfa05} to the case of different dimensions $d_1, d_2$ for the input and output spaces and allow slant $\lambda<1$. The proof largely follows the proof of \cite[Theorem 2.1]{Pfa05}.

Counting all the elements the $2d$ faces of the cube in $d$ dimensions with side $2K$, we denote  
\begin{equation}\label{eq:cube.f}
f(d, K) = \abs{\Set{z \in \Z^{d}}{\size{z} = K}}  = 2d \: (2K)^{d-1} = d \: 2^{d} \: K^{d-1}.
\end{equation}

\begin{theorem}\label{thm:my.bi-infinite}
Let $1\leq p_1, p_2 \leq \infty$, $\frac{1}{p_1} + \frac{1}{q_1} = 1$.  
Let $M = [M_{z_2, z_1}]: \ell^{p_1}(\Z^{d_1}) \to \ell^{p_2}(\Z^{d_2})$ be a complex bi-infinite matrix such that for some $\lambda> \pambda>0$, we have 
\begin{enumerate}[(i)]

\item \label{item:ii} the entries of $M$ decay away from the $\lambda$-slanted diagonal, that is, 
\begin{equation*} 
\abs{M_{z_2, z_1}} \leq w(\lambda \norm{z_2}_\infty - \norm{z_1}_\infty) \:  (1+\norm{z_1}_\infty)^{r_1} \: (1+\norm{z_2}_\infty)^{r_2},
\end{equation*}
for some decaying function $w = \littleo{\abs{x}^{-r}}$, where $r > \frac{d_1}{q_1} + \frac{d_2}{p_2} +  r_1 + r_2$, and the positive constants $r_1, r_2 \in \R_+$ satisfy
\[ 
r > \frac{d_1}{q_1},\text{ and } r > r_2 - \frac{1}{q_1} + \frac{d_1}{q_1} + \frac{d_2}{p_2},
\]
\item \label{item:iii} and for any $N_0 \in \N$ there exists $N_2 > N_0$ such that the submatrix based on a $\lambda'$-slanted diagonal
\[ 
\Mtilde= \brackets[\Big]{M_{z_2, z_1} \colon \size{z_1} \leq \pambda N_2,  \ \size{z_2} \leq N_2},
\]
is singular.

\end{enumerate}
then $M$ has no bounded left inverses. 
In fact, for any $\eps>0$, there exists a compactly supported vector $x\in\ell^{p_1}(\Z^{d_1})$ such that $\norm{x}_{p_1} = 1$ and $\norm{M x}_{p_2}<\eps$. 
\end{theorem}
\begin{figure}[ht]
\centering
\includegraphics{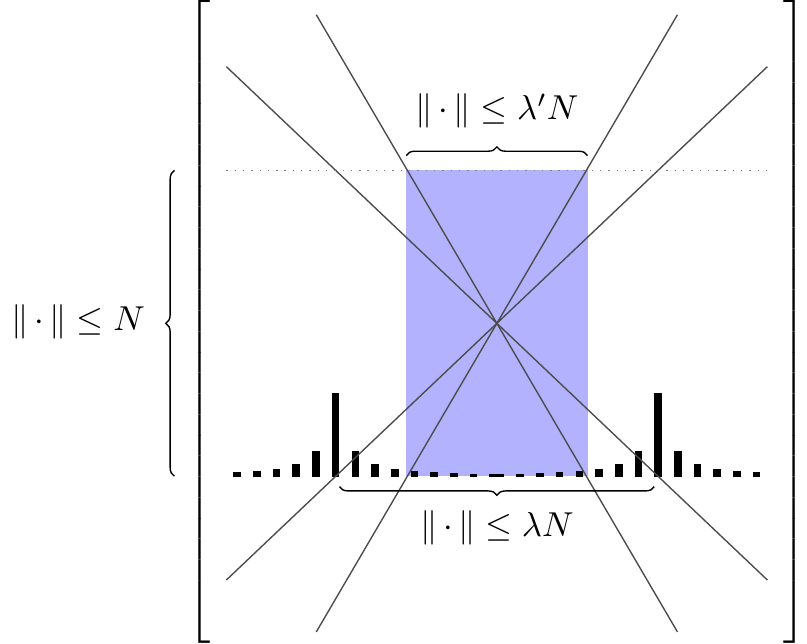}
\caption{A $\pambda$-slanted principal minor $M_{[N, \pambda N]}$ of a matrix exhibiting a decay away from the $\lambda$-slanted diagonal. } 
\end{figure}

\begin{proof}
Fix some $N_2$ to be chosen later, and $N_1 = \floor{\pambda N_2}$ in such a way that a submatrix of $M$ given in \eqref{item:iii} has a non-trivial kernel. 
Let a vector $\tilde{x}$ be such that $\norm{\tilde{x}}_{p_1} = 1$, and $\Mtilde{\tilde{x}} = 0$. 
We define $x \in \ell^{p_1}(\Z^{d_1})$ by padding with zeros: $x_{z_1} = \begin{dcases}
\tilde{x}_{z_1}, &\size{z_1}\leq N_1,\\
0, &\size{z_1} > N_1. 
\end{dcases}$
\begin{figure}[ht]
\centering
\includegraphics{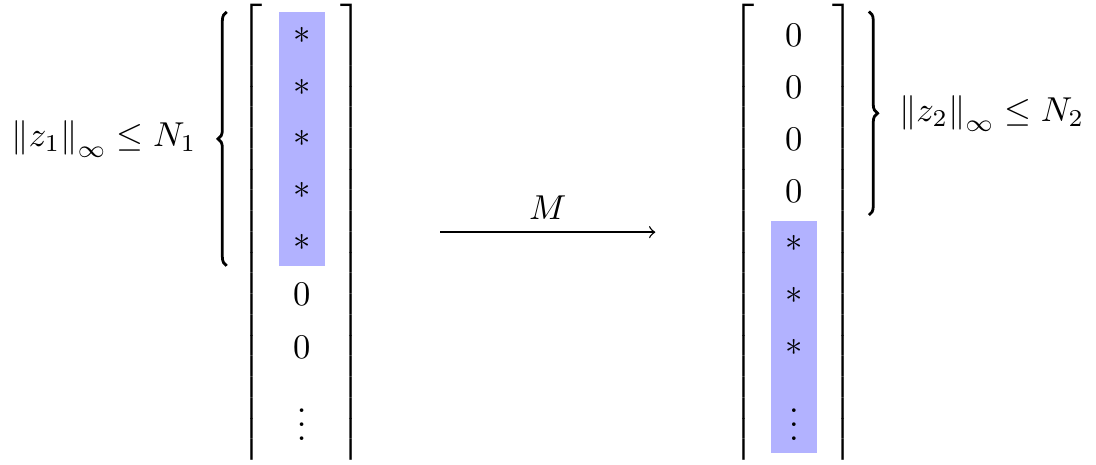}
\caption{The positions of non-zero components in $x$ and $Mx$.}
\end{figure}

Clearly, $(Mx)_{z_2} = 0$ for all $z_2$ such that $\norm{z_2}_\infty \leq N_2$. We now estimate $(Mx)_{z_2}$ for all $\norm{z_2} > N_2$. 
Fix such a $z_2$. 
\begin{align*}
\abs{(Mx)_{z_2}}^{q_1} &= \abs{\sum_{\size{z_1}\leq N_1} M_{z_2, z_1} x_{z_1}}^{q_1} \\
&\leq \norm{x}_{p_1}^{q_1} \sum_{\size{z_1}\leq N_1} \abs{M_{z_2, z_1}}^{q_1} \\
&\leq \coeff{z_2}^{q_1 r_2} \sum_{\size{z_1}\leq N_1 } \coeff{z_1}^{q_1 r_1} \: w(\underbrace{\lambda \size{z_2} - \size{z_1}}_{k_1})^{q_1} \\
&\leq \coeff{z_2}^{q_1 r_2} \: (N_1+1)^{q_1 r_1}  \sum_{\substack{\lambda \size{z_2}-N_1 \leq k_1 \\ k_1 \leq \lambda \size{z_2}}} \: w(k_1)^{q_1} \: f(d_1, k_1).
\end{align*}
where for each $\size{z_1}$ the number  of summands is $f(d_1, k_1)$, $f$ given in \eqref{eq:cube.f}. 
We can now  compute 
\begin{equation}\label{eq:Mxp2}\begin{split}
\norm{Mx}_{p_2}^{p_2} &= \sum_{\size{z_2}>N_2} \abs{(Mx)_{z_2}}^{p_2} \\
&\leq \sum_{\size{z_2}>N_2} \coeff{z_2}^{p_2 r_2} \:  (N_1+1)^{p_2 r_1}  \\
&\qquad \times \paren*{ \sum_{\lambda \size{z_2}-N_1 \leq k_1 \leq \lambda \size{z_2}}  \: w(k_1)^{q_1} \: f(d_1,k_1)  } ^{\frac{p_2}{q_1}} \\
&\leq d_2 \: 2^{d_2} \:   (N_1+1)^{p_2 r_1} \:  (\alpha \: d_1 \: 2^{d_1}  )^\frac{p_2}{q_1}  \\
&\qquad \times \sum_{k_2>N_2} (k_2+1)^{p_2 r_2 + d_2 - 1} \: \paren[\Big]{ \sum_{\lambda k_2 -N_1 \leq k_1\leq \lambda k_2}  \: w(k_1)^{q_1} \: k_1^{d_1-1} }^{\frac{p_2}{q_1}},
\end{split}\end{equation}
 where we denoted $k_2 = \size{z_2}$ and used the definition of $f$ in \eqref{eq:cube.f}.
Consider $\lambda k_2 - k_1 > \lambda N_2 -  \pambda N_2 > 1$ for $N_2 > N_0 = 2\ceiling{(\lambda-\pambda)^{-1}}$. We estimate the sums with the integrals, letting $x = k_2+1$ and $y = k_1$.
\begin{align*}
\MoveEqLeft \sum_{k_2>N_2} (k_2+1)^{p_2 r_2 + d_2 - 1} \: \paren[\Big]{ \sum_{\lambda k_2 -N_1 \leq k_1\leq \lambda k_2}  \: w(k_1)^{q_1} \: k_1^{d_1-1}  } ^{\frac{p_2}{q_1}}  \\
&\leq \int_{N_2 + 1}^\infty x^{p_2 r_2 + d_2 - 1} \: \paren[\Big]{ \int_{\lambda (x-1) - N_1-1}^{\lambda (x-1)}   \: w(y)^{q_1} \: y^{d_1-1}  \d y} ^{\frac{p_2}{q_1}}  \d x \\
\shortintertext{ let $w(y) = v(y)\: y^{-r}$ such that $v(y) \in C_0(\R)$ and $s = q_1 r - d_1 >0$  }
&\leq \int_{N_2 + 1}^\infty x^{p_2 r_2 + d_2 - 1} \: \paren[\Big]{ \int_{\lambda (x-1) - N_1-1}^{\lambda (x-1)}   \: v(y)^{q_1} \: y^{-s-1}  \d y} ^{\frac{p_2}{q_1}}  \d x \\
&\leq  \norm{v}_\infty^{p_2} \int_{N_2+1}^\infty x^{p_2 r_2 + d_2 -1} \paren[\Big]{ \int_{\lambda (x-1) - N_1-1}^{\lambda (x-1)} y^{-s-1}  \d y}^{\frac{p_2}{q_1}} \d x \\
&\leq \norm{v}_\infty^{p_2} \int_{N_2+1}^\infty x^{p_2 r_2 + d_2 -1} \paren[\Big]{ \frac{1}{s} \: y^{-s}\eval_{\lambda (x-1)}^{\lambda (x-1) - N_1-1}}^{\frac{p_2}{q_1}} \d x.
\intertext{Apply \autoref{lem:convex}, let $C = C(v, \lambda, \pambda, r, p_2, q_1, d_1, d_2)$, possibly different on different lines, and let $B = -p_2 r_2 - d_2  +  \frac{p_2}{q_1}\paren*{s+1}$, also guaranteed to be positive by \eqref{item:ii}.  
Simplify }
&\leq C \: \int_{N_2+1}^\infty x^{p_2 r_2 + d_2 -1} \paren[\Big]{ (N_1+1)\: x^{-s-1}   }^{\frac{p_2}{q_1}} \d x \\
&= C\: (N_1+1)^{\frac{p_2}{q_1}}  \int_{N_2+1}^\infty x^{-B-1}  \d x \\
&= C\: (N_1+1)^{\frac{p_2}{q_1}}  (N_2+1)^{-B}.
\end{align*}
We remember $N_1 = \floor{\pambda N_2}$ and finally return to estimate \eqref{eq:Mxp2}
\begin{align*}
\norm{M x}_{p_2} = C (N_1+1)^{r_1 + \frac{1}{q_1}} (N_2+1)^{-\frac{B}{p_2}} 
\leq  C \:  (N_2+1)^{r_1 + \frac{1}{q_1} -\frac{B}{p_2}},
\end{align*}
where we used again that $N_2>N_0$. For any fixed $\eps$ we can now find $N_2$ large enough so that $\norm{Mx}_{p_2} <\eps$, because by \eqref{item:ii},
\[ 
r_1 + \frac{1}{q_1} -\frac{B}{p_2} = r_1 + r_2 + \frac{d_2}{p_2} + \frac{d_1}{q_1} - r  < 0. 
\]
\end{proof}

\begin{lemma}\label{lem:convex}
Let $N_2 \geq N_0 = 2\ceiling{(\lambda-\pambda)^{-1}}$ and $N_1 = \floor{\pambda N_2}$ with $\pambda < \lambda$. For $x\geq N_2+1$, we estimate 
\[
(\lambda(x-1)-N_1-1)^{-r} - (\lambda(x-1))^{-r} \leq C(\lambda, \pambda, r) \: x^{-r-1} \: (N_1+1).
\]
\end{lemma}
\begin{proof}
Denote for brevity $a = \lambda(x-1) \geq \lambda N_2$ and $b = N_1+1$. Let $t = b/a$.
Observe that for $N>N_0 = 2\ceiling{(\lambda-\pambda)^{-1}} > \frac{2}{\lambda-\pambda}$, 
\[ 
t = \frac{b}{a}  \leq \frac{\pambda N_2 + 1}{\lambda N_2}  = \frac{\pambda}{\lambda} + \frac{1}{\lambda N_2} \leq \frac{\pambda}{\lambda} + \frac{\lambda-\pambda}{2\lambda} = 
\frac{\lambda+\pambda}{2\lambda} \coloneqq t_0 < 1. 
\]
Consider the convex function $f(t) = (1-t)^{-r}-1$ for $t\in [0, 1)$. Since $f(0)=0$, we can estimate 
\[ 
f(t) \leq \frac{f(t_0)}{t_0} t = C(\lambda, \pambda) \: t \quad \text{ for all } t \in [0,t_0]. 
\]
We can now bound the original quantity 
\[
(a-b)^{-r} - a^{-r} = a^{-r} \paren{ \paren*{1-t}^{-r} - 1} \leq a^{-r} \: C \: t = C' \: (x-1)^{-r-1} \: (N_1+1). 
\]
It remains to observe that for $x\geq N_2+1$, we have $(x-1)^{-r-1} \leq x^{-r-1}\paren*{\frac{x-1}{x}}^{-r-1}$, and 
\[
\paren*{\frac{x-1}{x}}^{-r-1} \leq \paren*{1 - \frac{1}{x}}^{-r-1} \leq \paren*{1- \frac{1}{N_2}}^{-r-1} \leq \paren*{\frac{\lambda+\pambda}{2}}^{-r-1} = C(\lambda, \pambda, r),
\]
which can be also absorbed into the constant. 
\end{proof}
\bibliographystyle{elsarticle-num}
\bibliography{PfaZh02}
\end{document}